\documentclass[a4paper,twoside]{article}
\usepackage{a4}

\usepackage{hyperref}     
\hypersetup{
	colorlinks=true,        
	linkcolor=blue,         
	citecolor=blue,
	urlcolor=blue,
}

\usepackage{graphicx}
\usepackage{amssymb}
\usepackage{epstopdf}
\usepackage{amsmath,amsthm}
\numberwithin{equation}{section}
\numberwithin{figure}{section}

\usepackage{upref}
\usepackage[active]{srcltx}
\usepackage{xypic}

\usepackage{amsthm,amsmath,amssymb}

\usepackage{mathrsfs}

\usepackage{graphicx} 
\usepackage{amsmath, bm}

\theoremstyle{plain}
\newtheorem{thm}{Theorem}[section]

\newtheorem{lem}[thm]{Lemma}

\newtheorem{definition}[thm]{Definition}

\newtheorem{remark}[thm]{Remark}

\begin{document}
	\title{Contact lifts and  H\"older lifts to central extension of Carnot groups}
	\author{Yihan Cui}
	\maketitle
	\begin{abstract}
		We consider the existence problem of lift $F$ of a map $f$ between Carnot group with different smoothness, where we use central extension to define lifting. Our main result is the existence of the contact lifts of Lipschitz and Sobolev maps and the rigidity result for the contact lift of quasiconformal maps: a quasiconformal map admits a contact lift then it is bi-Lipschitz. We also show a necessary criterion for the extension of $\gamma$-H\"older lift when $\gamma>\frac{1}{2}$ for step-$n$ Carnot group.    
	\end{abstract}
	\section{Introduction}
	The lifting problem is a useful tool to study and construct a contact map between higher step Carnot group. Let's illustrate this problem.
	
	Let $V_1 \xrightarrow{\iota} G_1 \xrightarrow{\pi_1} H_1$ and $V_2 \xrightarrow{\iota} G_2 \xrightarrow{\pi_2} H_2$ be central extensions of Carnot group. A typical example is the Heisenberg group $R \xrightarrow{\iota} \mathbb{H}_n \xrightarrow{\pi} \mathbb{R}^{2n}$. The purpose is to find a contact lift $F:G_1\to G_2$ of contact map $f:H_1\to H_2$ such that $\pi_2 \circ F=f\circ \pi_1$ i.e. the following diagram commutes
	
	\begin{displaymath}
		\xymatrix{
			G_1 \ar[d]^{F} \ar[r]^{\pi_1} & H_1 \ar[d]^{f}
			\\
			G_2  \ar[r]^{\pi_2} & H_2
		}
	\end{displaymath}
	In the Heisenberg group case, we want to find a contact lift $F$ of $f$ such that the following diagram commutes
	
	\begin{displaymath}
		\xymatrix{
			\mathbb{H}_n \ar[d]^{F} \ar[r]^{\pi} & \mathbb{R}^{2n} \ar[d]^{f}
			\\
			\mathbb{H}_n  \ar[r]^{\pi} & \mathbb{R}^{2n}
		}
	\end{displaymath}
	
	When the domain is one dimensional, by control theory (\cite{ABB2019}, Section 8), a curve lays in the Euclidean space can be uniquely lifted into Carnot group. When the domain is higher dimensional, we are supposed to consider the integrability of submanifold.
	
	A well studied topic is the Heisenberg group, a classical technique used by Allcock is the lift of a map $\mathbb{D}\to \mathbb{R}^{2n}$ in \cite{Allcock1998}. He construct a map $\mathbb{D}\to \mathbb{H}_n$ satisfying a Lagrangian condition.
	
	For map $\mathbb{R}^k\to \mathbb{R}^k$, a type of problems is the existence of embedding from a domain in Euclidean space into Heisenberg group. Because of the Lie algebra graded structure, the dimension of the domain must be smaller or equal to $n$. For more detail, readers can go to \cite{FSSC2007}. A useful way to avoid the restriction of Lie algebra homomorphism is to study maps with poor smoothness condition like H\"older maps. In \cite{HMS2025}, Piotr Haj{\l}asz, Jacob Mirra and Armin Schikorra study the Gromov non-embedding theorem with differential form which is a crucial tool in our paper for studying the existence of H\"older lifts and their properties.
	
	\subsection*{Historical development of the lifting problem}
	The lifting problem did not originate as an isolated question about central extensions. It grew out of three decades of interaction between rigidity theory, contact geometry and low-regularity analysis on Carnot groups, and it is useful to keep this genealogy in mind, because the obstructions one meets at each level of regularity are direct descendants of the techniques developed in the corresponding period.

	The geometric source of the problem is Mostow's strong rigidity theorem \cite{Mostow1973}, which forces a quasi-isometry between rank-one symmetric spaces to be a bounded distance from an isometry, and therefore forces the induced boundary map to be far more regular than a priori expected. Pansu made the boundary side of this picture precise in his foundational paper \cite{Pansu1989}: he introduced Carnot--Carath\'eodory metrics on the boundaries of rank-one symmetric spaces, showed that quasi-isometries between such spaces correspond to quasisymmetric (in particular quasiconformal) maps between their ideal boundaries, and introduced the derivative that now bears his name. In parallel, his conformal-dimension paper \cite{Pansu1989b} supplied the invariant that distinguishes the boundaries quantitatively. The Pansu differential is the first appearance of the phenomenon that governs this whole paper: the infinitesimal object attached to a low-regularity map between Carnot groups is not an arbitrary linear map, but a \emph{graded Lie algebra homomorphism}.
	
	The model case is the Heisenberg group. A classical result is that the Heisenberg group $\mathbb{H}_n$ is the ideal boundary of a Siegel domain in $\mathbb{C}^{n+1}$ under the hyperbolic metric induced by the Bergman metric (\cite{Krantz2017}, Section 4); see also \cite{Goldman1999} for the complex-hyperbolic viewpoint. Quasi-isometries between such spaces correspond naturally to quasisymmetries between their ideal boundaries \cite{Pansu1989}. This is the reason why the lifting problem for $\mathbb{R}^{2n}\to\mathbb{R}^{2n}$, which is the case treated in Sections~\ref{sec-approx} and~\ref{sec-holder-lift} below, was originally motivated by Mostow rigidity: one wants to know which boundary maps can be realised as the horizontal part of a contact map one step higher in the central extension.

	The first systematic analytic treatment of contact maps on the Heisenberg group is due to Kor\'anyi and Reimann \cite{KoranyiReimann1985}: a smooth quasiconformal map satisfies $f^{*}\omega=\lambda\omega$ for the standard contact form, and it is generated by a flow of a contact vector field of the form $v=-\frac{1}{4}(Yp)X+\frac{1}{4}(Xp)Y+pT$. Their later foundational paper \cite{KoranyiReimann1995} built the Beltrami-type apparatus for the theory. This flow method is the direct ancestor of the vector-field generation technique we use in Example~\ref{ex-filiform} for the filiform group $F^{3}$.
	
	Two further constructions from this period shape the lifting problem in higher dimension. Capogna--Tang \cite{CapognaTang1995} study lifts of symplectomorphisms $\mathbb{R}^{2n} \to \mathbb{R}^{2n}$ to full-rank contact maps between Heisenberg groups $\mathbb{H}_n$, in the context of uniform and NTA domains; this is the classical source for the statement that a symplectic map on the base admits a contact lift on the total space. Allcock \cite{Allcock1998}, in his proof of the quadratic isoperimetric inequality for $\mathbb{H}_n$ ($n\ge 2$), developed the explicit filling of a disk that underlies the Lagrangian condition: to lift a map $\mathbb{D}\to\mathbb{R}^{2n}$ to $\mathbb{D}\to\mathbb{H}_n$ one must control the symplectic area enclosed by the projected curve, and it is the filling that supplies this control.

	The next step was to understand that the image of a lift must carry the correct intrinsic dimension. Franchi--Serapioni--Serra Cassano \cite{FSSC2007} developed the theory of intrinsic regular submanifolds, intrinsic graphs and the associated area formula in Heisenberg groups, and this is the correct framework for the dimensional restriction mentioned above: for low-dimensional objects the relevant submanifolds are Legendrian, and the graded structure of the Lie algebra forbids a Euclidean domain of dimension $k>n$ from carrying such a structure.
	
	In the same decade, Gromov initiated the systematic study of H\"older maps into sub-Riemannian manifolds. In \cite{Gromov1996} (see also \cite{Gromov1999}) he proved the non-embedding theorem that every $\gamma$-H\"older embedding $\mathbb{R}^{m}\to\mathbb{H}_{n}$ with $m\ge n+1$ satisfies $\gamma\le\frac{n+1}{n+2}$, and he conjectured that for $\mathbb{R}^{2}\to\mathbb{H}^{1}$ the sharp threshold should be $\gamma\le\frac{1}{2}$. This conjecture is the reason why the exponent $\frac{1}{2}$ appears throughout the present paper, and it remains one of the central open problems in the area.
	
	The earliest work devoted specifically to the existence of contact lifts is due to Balogh--Hoefer-Isenegger--Tyson \cite{BHIT2006}, who considered the existence of contact lifts of Lipschitz maps $\mathbb{R}^2\to \mathbb{R}^2$ into the first Heisenberg group, in connection with horizontal fractals and the iterative function systems that produce them. Their methods provide a good example when studying the existence problem of contact lifts and even H\"older lifts. A particularly instructive outcome of their work is that a lift can be of bounded variation along horizontal curves without being horizontally Lipschitz; in other words, at low regularity the obstruction is genuinely geometric and cannot be read off from the classical Euclidean derivative.

	A decisive conceptual advance was the recognition that the correct cohomological object on a Carnot group is not the de Rham complex but the Rumin complex. Rumin \cite{Rumin1990,Rumin1994} constructed a subcomplex of the de Rham complex, adapted to the stratification, which resolves the constant sheaf and in which the exterior differential is replaced by a weighted, in general higher-order, operator; see \cite{PansuRumin2018} for the $\ell^{q,p}$-cohomological development and \cite{Tripaldi2020} for the nilpotent case. Because the Rumin differential only sees forms of prescribed weight, it is precisely the tool that detects whether a $2$-cocycle defining a central extension is cohomologically trivial in the graded sense.
	
	Kleiner--M\"uller--Xie \cite{KMX2021} then proved that the Rumin complex is respected by Sobolev mappings: the Pansu pullback induces a chain map between the smooth Rumin complex and the distributional one, and consequently the Rumin flat complex is bi-Lipschitz invariant. This result is what makes it possible to pass the smooth cohomology class condition to the Sobolev category, and it is the technical reason for the hypothesis $p>Q_{1}$ that appears in Lemma~\ref{lem-Sobolev-intro} below.
	
	In the same period, the topological side of the problem was clarified through Lipschitz homotopy groups. Dejarnette--Haj{\l}asz--Lukyanenko--Tyson \cite{DHLT2014} showed that Lipschitz maps are not dense in $W^{1,p}(M,\mathbb{H}_{n})$ when $\dim M\ge n+1$ and $n\le p<n+1$, and interpreted the obstruction as the non-triviality of the $n$-th Lipschitz homotopy group; Wenger--Young \cite{WengerYoung2014} subsequently showed that Lipschitz and smooth horizontal homotopy groups may differ. These results explain why, in the present paper, closed horizontal curves---rather than more general cycles---carry the entire obstruction to the existence of a lift.
	
	\subsection*{Recent developments }
	The smooth case has very recently been settled. In \cite{HHI2025}, Eero Hakavuori, Susanna Heikkil\"a and Toni Ikonen study the existence problem of lifting a smooth contact map between Carnot groups. Their criterion is expressed in terms of the Rumin complex $E_{0}$ of the base group: writing $\alpha_{1}$ and $\beta_{2}$ for the relevant potentials of the two $2$-cocycles and $A:V_{1}\to V_{2}$ for the induced map on the fibres, a smooth contact map $f$ admits a smooth contact lift if and only if $A\circ\alpha_{1}-f^{*}\beta_{2}$ is closed in the Rumin sense, that is, if and only if $f^{*}\rho_{2}$ and $A\circ\rho_{1}$ lie in the same Lie algebra cohomology class. The necessity of the Pansu-pullback condition is unconditional, and it is also sufficient when the cocycle $\rho_{2}$ sits in the maximal weight, because the Rumin differential then has no room to produce a non-trivial coboundary. We record their smooth criterion as Theorem~\ref{thm-smooth-criterion} and use it as the model for everything that follows.

	The parallel development in the H\"older category has been equally rapid. In \cite{HMS2025}, Haj{\l}asz, Mirra and Schikorra construct distributional pullbacks, distributional Jacobians, a change-of-variables formula and a Stokes theorem for H\"older continuous maps into Heisenberg groups, and they give a far-reaching generalisation of Gromov's non-embedding theorem. This is a crucial tool in our paper for studying the existence of H\"older lifts and their properties.
	
	Two closely related advances should be mentioned, because together they delimit exactly what one can hope for. First, Z\"ust \cite{Zuest2016} solved Gromov's H\"older equivalence problem for the first Heisenberg group: a map of class $C^{\alpha}$, $\alpha>\frac{1}{2}$, from a quasi-convex space with vanishing first Lipschitz homology into $\mathbb{H}$ factors through a tree, and hence cannot be injective on a domain of dimension $>1$. Second, Haj{\l}asz and Schikorra \cite{HajlaszSchikorra2023} gave an elementary proof of a generalisation of Gromov's non-embedding theorem, avoiding the $h$-principle. Complementing these, Balogh--Haj{\l}asz--Wildrick \cite{BHW2014} proved that a map $f:\Omega\to\mathbb{H}_{n}$ which is $\alpha$-H\"older with $\alpha>\frac{1}{2}$ and locally Lipschitz as a map into $\mathbb{R}^{2n+1}$ cannot be injective, and they answered a question of Magnani \cite{Magnani2013} on the rank of the weak derivative of a weakly contact Sobolev map.
	
	Very recently, Schikorra \cite{Schikorra2026} has shown by convex integration that the strong form of Gromov's conjecture cannot hold in all dimensions: for every $n\ge 1$ any smooth horizontal map $\mathbb{S}^{n}\to\mathbb{H}_{n}$ extends to a $C^{\alpha}$-map of $\mathbb{B}^{n+1}$ for some $\alpha>\frac{1}{2}$, and for $n\ge 3$ there are $C^{\alpha}$-embeddings $\mathbb{B}^{n+1}\hookrightarrow\mathbb{H}_{n}$ with $\alpha>\frac{1}{2}$. This is a vivid reminder that the thresholds appearing in our theorems are thresholds for the \emph{lifting} problem, not thresholds for embeddability, and the two should not be conflated.

	\subsection*{Structure of the present paper}
	Follow the work of \cite{HHI2025}, we consider the existence problem of lift $F$ of a map $f$ between Carnot groups with different smoothness. Let $V_1 \xrightarrow{\iota} G_1 \xrightarrow{\pi_1} H_1$ and $V_2 \xrightarrow{\iota} G_2 \xrightarrow{\pi} H_2$ be fixed central extensions of Carnot groups given by 2-cocycles $\rho_{1}$ and $\rho_{2}$, see Section 2 for details.
	
	If we just want to find a map $F$ such that the diagram commutes is easy, the key is to find a lift $F$ with the same smoothness of $f$. An important and enlightening result is given by \cite{HHI2025}, whose criterion for the existence of smooth contact lifts is proved with the help of the Rumin complex. Since their argument is cohomological, we can reuse it for the existence of Lipschitz or Sobolev contact lifts as long as we can get a similar result like Lemma 5.3 in \cite{HHI2025}.
	
	Therefore, for Sections~\ref{sec-Lipschitz} and~\ref{sec-Sobolev}, our purpose is to generalise the result to the Lipschitz and Sobolev categories. The bridge is supplied by the observation that a closed horizontal curve whose lift is closed contributes zero circulation, so that the whole obstruction is carried by the periods of the fibre $1$-form.
	
	\begin{lem}\label{lem-Lipschitz-intro}
		Let $U_1 \subset \mathbb{H}_1$ and $\tilde{U}_1 \subset \pi_1^{-1}(U_1) \subset \mathbb{G}_1$ be domains. A Lipschitz contact map $f: U_1 \to \mathbb{H}_2$ admits a Lipschitz contact lift $F: \tilde{U}_1 \to \mathbb{G}_2$ if and only if
		\begin{equation}
			(f \circ \pi_1)\left( \Gamma_{\text{LIP}}(g, \tilde{U}_1) \right) \subset \Gamma_{\text{LIP}}^{\rho}(f(\pi_1(g)), \mathbb{H}_2)
		\end{equation}
		for some $g \in \tilde{U}_1$.
	\end{lem}
	Passing from Lipschitz to Sobolev regularity costs almost nothing in the formulation---the closed-curve condition remains the same---but the proof has to be redone, because the pullback $(f\circ\pi_{1})^{*}\alpha_{2}$ is now only defined along almost every curve. The necessary analytic input is the absolute continuity on almost every curve together with the upper-gradient characterisation of $W^{1,p}$, and the homogeneous dimension $Q_{1}$ enters through the a.e. representative. This yields the companion statement.
	
	\begin{lem}\label{lem-Sobolev-intro}
		Let $U_1 \subset \mathbb{H}_1$ and $\tilde{U}_1 \subset \pi_1^{-1}(U_1) \subset \mathbb{G}_1$ be domains. A Sobolev contact map $f: U_1 \to \mathbb{H}_2$ admits a Sobolev contact lift $F: \tilde{U}_1 \to \mathbb{G}_2$ if and only if
		\begin{equation}
			(f \circ \pi_1)\left( \Gamma_{\text{LIP}}(g, \tilde{U}_1) \right) \subset \Gamma_{W^{1,p}}^{\rho}(f(\pi_1(g)), \mathbb{H}_2) \ \text{for}\  \mathcal{H}^{Q_1-1}-a.e.
		\end{equation}
		for some $g \in \tilde{U}_1$, where $f,F\in W^{1,p}$ and $p>Q_1$, $Q_1$ is homogeneous dimension of $H_1 $.
	\end{lem}
	In these two lemmas, we use $F(\widetilde{\gamma}(1))=f(\gamma(1)) \cdot exp(\int_{\gamma}\left(f \circ \pi_{1}\right)^{*} \alpha_{2})$
	to define the lifting $F$ of $f$, since (1.1) and (1.2) leads to $F$ is well defined in this way. For the smoothness of lifting $F$, we find the upper gradient of $F$ with the help of this definition.
	
	Once the existence theory is in place, it is natural to ask how much freedom is left. The answer is that surprisingly little: the very structure that makes the lift possible also rigidifies the base map. In Section~\ref{sec-rigidity}, we study the rigidity of the contact lift of quasiconformal maps getting the following result.
	
	\begin{thm}\label{thm-biLipschitz-intro}
		Suppose that $f$ is a $K$-quasiconformal contact map and it admits a contact lift $F$ then $f$ is bi-Lipschitz.
	\end{thm}
	Because of the restriction of the structure of lifting, the choice of maps are greatly limited. With the help lift invariant differential form, we can use the concept of weight to distinguish the different actions of different layers in the Lie algebra. To be specific, the $i$-th layer are expending at the rate of $r^i$ when the horizontal layer are are expending at the rate of $r$.
	
	A interesting case is the quaternionic Heisenberg group $\mathbb{H}_\mathbb{H}^n$, because the lifting is from a map form $\mathbb{R}^{4n}$ to $\mathbb{R}^{4n}$, the second layer provides a very strong restriction for the algebra structure on Lie algebra. The mechanism is the same as in the Liouville-type theorems of Capogna--Cowling \cite{CapognaCowling2006} and Ottazzi--Warhurst \cite{OttazziWarhurst2011}: a contact map whose Pansu differential preserves three anticommuting complex structures is forced into a finite-dimensional group.
	
	\begin{thm}\label{thm-quaternionic-intro}
		Suppose $f$ is a Sobolev continuous map on $\mathbb{R}^{4n}$ which admits a contact lift $F$ on $\mathbb{H}_\mathbb{H}^n$ and $D_PF$ is full-rank for every point, then $f$ is an affine transformation.
	\end{thm}
	
	In Section~\ref{sec-approx}, we show a classical example Heisenberg group $\mathbb{H}^n$. Thanks to the standard symplectic form on $\mathbb{R}^{2n}$, we can find a family of smooth lifting to approximate the sobolev lifting. The same technique applies verbatim to H\"older lifts on $\mathbb{H}^{1}$, yielding Theorem~\ref{thm-holder-approx-body}.
	
	\begin{thm}\label{thm-approx-intro}
		For any sobolev map $f$ on $\mathbb{R}^{2n}$ and its contact lift $F$ on $\mathbb{H}^n$, there are a serious of smooth $g_\varepsilon$ which admit smooth contact lift $G_\varepsilon$, $g_\varepsilon\rightarrow f$ and $G_\varepsilon\rightarrow F$ as $\varepsilon\rightarrow 0$.
	\end{thm}
	
	In Section~\ref{sec-holder-lift}, we study the H\"older lift. The first ingredient is the anisotropic decay estimate already mentioned above, which we state here in the general form; it is the exact replacement, in the H\"older category, for the statement that the Pansu differential is a graded Lie algebra homomorphism.

	\begin{thm}\label{thm-holder-lift-intro}
		Let $U_1 \subset \mathbb{H}$ and $\tilde{U}_1 \subset \pi_1^{-1}(U_1) \subset \mathbb{G}_1$ be domains and $rank(\mathbb{G}_2)=rank(\mathbb{H}_2)$. For $\beta>\frac{1}{2}$, A H\"older map $f\in C^{0,\beta}$, $f: U_1 \to \mathbb{H}_2$ admits a H\"older lift $F\in C^{0,\beta}$, $F: \tilde{U}_1 \to \mathbb{G}_2$ if and only if
		\begin{equation}
			\int_{f \circ \pi_{1}\circ\widetilde{\gamma}} \alpha = 0
		\end{equation}
		for all Lipschitz closed curves $\widetilde{\gamma}$ on $\mathbb{G}_1$. We define $\int_{\widetilde{\gamma}} (f \circ \pi_{1})^*\alpha:=\lim_{\varepsilon\rightarrow 0}\int_{\widetilde{\gamma}} (f_\varepsilon \circ \pi_{1})^*\alpha$, where  $f_{\varepsilon}=f * \eta_{\varepsilon}=\int_{\mathbb{G}_1 } f(y) \eta_{\varepsilon}\left(y^{-1} x\right) d y$  and $\eta_{\varepsilon}$ is the heat kernel on Carnot group $\mathbb{G}_1$ .
	\end{thm}
	We generalize the definition of pullback of the differential form under $\beta$-H\"older map $f$. With it, we similarly define the lift of $f$ with the form of integral. In the proof, we take advantage of the the Lie algebra graded-structure-preserving property for H\"older map to show the smoothness of H\"older lift $F$. To be specific, if the in the horizontal layer the movement is approximately $\varepsilon^\beta$, the $i$-th layer would go $\varepsilon^{i\beta}$(including the extension part $exp(\int_{\gamma}\left(f \circ \pi_{1}\right)^{*} \alpha_{2})$), which matches the metric structure of Carnot group.
	
	The closed-curve condition of Theorem~\ref{thm-holder-lift-intro} is conceptually clean but hard to verify in practice, since it asks for information about all Lipschitz closed curves. We therefore also show a criterion that is easy to check for the existence of H\"older lift.
	
	\begin{thm}\label{thm-criterion-intro}
		Let $V_1 \to G_1 \to H_1$ and $V_2 \to G_2 \to H_2$ be central extensions, where $rank(H_2)=rank(G_2)$. Let $U$ is a simply connected domain in $H_1$ and $f: U \to H_2$ be a $\gamma$-H\"older map for $\beta>\frac{1}{2}$. Then $f$ admits a $\beta$-H\"older lift $F$ if and only if there is a linear map $\varphi: V_1 \to V_2$ and a $V_2$-value 1-form $\lambda$ with weight $\geq 2 $ such that the integral of $f^*\alpha_{2}$ on any horizontal curve equals to the integral of $\varphi\circ \alpha_1$ on the same curve, i.e.
		\begin{equation*}
			\varphi\circ \int_{\gamma} \alpha_1=  \int_{\gamma} f^*\alpha_{2}.
		\end{equation*}
	\end{thm}
	In theorem~\ref{thm-criterion-intro}, because of lack of differential structure, we can just study under the graded structure. Although it is not as elegant as Lie algebra homomorphism, there is a special case which is easy to prove. If the central extension is just for the highest layer, we can decompose the lifting $F$ into two part: $F(hk)=f(h)\Phi(k)$, where $h\in\mathbb{H}$, $k\in V_1$ and $\Phi:V_1\to V_2$.
	
	Section~\ref{sec-prelim} fixes notation and recalls the basic properties of Carnot groups, central extensions by $2$-cocycles, and H\"older spaces. Section~\ref{sec-Lipschitz} treats Lipschitz contact lifts and Section~\ref{sec-Sobolev} their Sobolev counterparts; in both cases the answer is a closed-horizontal-curve condition, and the fibre behaviour of a lift is described by a group homomorphism $\Phi:V_{1}\to V_{2}$. Section~\ref{sec-rigidity} contains the rigidity results for quasiconformal and quaternionic lifts. Section~\ref{sec-approx} constructs smooth approximations of Sobolev and H\"older lifts on the Heisenberg group via Moser's trick. Finally, Section~\ref{sec-holder-lift} develops the H\"older theory: the anisotropic decay estimate, the definition of the H\"older pullback, and the two existence criteria. For the reader's convenience, the smooth criterion of \cite{HHI2025} that motivates our approach is recorded as Theorem~\ref{thm-smooth-criterion}.
	
	\section{Preliminary}\label{sec-prelim}
	We recall the basic properties of Carnot groups. 
	We recall the basic properties of Carnot groups, referring the reader to \cite{Montgomery2002} and \cite{ABB2019} for a systematic treatment of sub-Riemannian geometry.
	A Lie algebra  $\mathfrak{g}$  is stratified if it has a decomposition  $\mathfrak{g}=\mathfrak{g}^{[1]} \oplus \cdots \oplus \mathfrak{g}^{[s]}$  with  $\left[\mathfrak{g}^{[1]}, \mathfrak{g}^{[k]}\right]=\mathfrak{g}^{[k+1]}$  for all $ k \geq 1$ , where we denote  $\mathfrak{g}^{[k]}=\{0\}$  for  $k \geq s+1 $. A (sub-Riemannian) Carnot group is a simply connected nilpotent Lie group $ G$  with a stratified Lie algebra  $\mathfrak{g}$  equipped with an inner product on the horizontal layer  $\mathfrak{g}^{[1]}$ . For convenience, We will need an inner product not only on the horizontal layer  $\mathfrak{g}^{[1]}$ , but instead on the whole Lie algebra  $\mathfrak{g}$ . We always equip the Lie algebra of a Carnot group with an inner product for which the layers  $\mathfrak{g}^{[1]} \oplus \cdots \oplus \mathfrak{g}^{[s]}$  are pairwise orthogonal. The homogeneous dimension of $ G$  is  $Q_{G}:=\sum_{k=1}^{s} k \operatorname{dim}\left(\mathfrak{g}^{[k]}\right) $. The (horizontal) rank of  $G$  is the dimension of the horizontal layer  $\mathfrak{g}^{[1]} $.
	
	Given $ g \in G$ , we denote by  $L_{g}: G \rightarrow G$  the left-translation  $L_{g}(h)= g h$ . We identify  $T_{e} G$  with  $\mathfrak{g}$ . An absolutely continuous curve  $\gamma:[0,1] \rightarrow G$  is horizontal if for almost every  $t \in[0,1] $ its left-trivialized derivative  $\left(L_{\gamma(t)}^{-1}\right) * \dot{\gamma}(t)$  is contained in the horizontal layer $ \mathfrak{g}^{[1]} $. The length of the horizontal curve  $\gamma$  is
	\begin{equation*}
		\ell(\gamma)=\int_{0}^{1}\left\|\left(L_{\gamma(t)}^{-1}\right)_{*} \dot{\gamma}(t)\right\| d t
	\end{equation*}
	where the norm is the one induced by the inner product on  $\mathfrak{g}^{[1]}$ . The sub-Riemannian distance between two points  $g, h \in G $ is
	\begin{equation*}
		d(g, h)=\inf \{\ell(\gamma) \mid \gamma:[0,1] \rightarrow G \text { horizontal, } \gamma(0)=g, \gamma(1)=h\} .
	\end{equation*}
	
	By construction, this distance is left-invariant.
	Carnot groups also admit a one-parameter group of automorphisms, known as the dilations  $\delta_{\lambda}: G \rightarrow G, \lambda>0 $. The associated Lie algebra automorphisms, also denoted  $\delta_{\lambda}: \mathfrak{g} \rightarrow \mathfrak{g} $, are the linear maps defined on each layer of the stratification as
	\begin{equation*}
		\delta_{\lambda}(X)=\lambda^{k} X, \quad X \in \mathfrak{g}^{[k]}, \quad 1 \leq k \leq s .
	\end{equation*}
	The sub-Riemannian distance is 1 -homogeneous with respect to these dilations, i.e.,  $d\left(\delta_{\lambda} g, \delta_{\lambda} h\right)=\lambda d(g, h) $.
	
	If we want to study the properties of contact lift, we need to start with central extension and related concepts.
	Let  $\mathfrak{h}$  be a Lie algebra,  $V$  a vector space, and  $\rho: \bigwedge^{k} \mathfrak{h} \rightarrow V$  a vector-valued  $k$ -form. Here, and in what follows, $ V$  is assumed to be a finite-dimensional vector space. The Lie algebra differential  $d_{0} \rho$  of  $\rho$  is the vector-valued  $(k+1)$ -form
	\begin{equation*}
		\begin{array}{l}
			d_{0} \rho\left(X_{1}, \ldots, X_{k+1}\right) \\
			\quad=\sum_{i<j}(-1)^{i+j} \rho\left(\left[X_{i}, X_{j}\right], X_{1}, \ldots, \hat{X}_{i}, \ldots, \hat{X}_{j}, \ldots, X_{k+1}\right),
		\end{array}
	\end{equation*}
	where  $\hat{X}_{i}$  means that  $X_{i}$  is omitted from the list. We identify  $\Lambda^{1} \mathfrak{h} \simeq \mathfrak{h}$ , so vector valued 1 -forms are identified with linear maps  $\mathfrak{h} \rightarrow V$ . A  $k$  cocycle is a  $k$ -form whose Lie algebra differential vanishes.
	
	\begin{definition}
		Let $\mathfrak{h}$ be a Lie algebra and $\rho \colon \bigwedge^{2}\mathfrak{h} \to V$ a $2$-cocycle with values in a vector space $V$. Denote the Lie bracket of $\mathfrak{h}$ by $[\cdot,\cdot]_{\mathfrak{h}}$. The $\emph{central extension}$ of $\mathfrak{h}$ by $\rho$ is a Lie algebra $\mathfrak{g}$ with a direct sum decomposition $\mathfrak{h} \oplus V$ equipped with the Lie bracket
		\[
		[X + A, Y + B]_{\mathfrak{g}} = [X, Y]_{\mathfrak{h}} + \rho(X, Y), \quad \text{for } X, Y \in \mathfrak{h},\ A, B \in V.
		\]
		The direct sum decomposition $\mathfrak{g} = \mathfrak{h} \oplus V$ induces a natural inclusion $\iota_{*} \colon V \to \mathfrak{g}$ and a projection $\pi_{*} \colon \mathfrak{g} \to \mathfrak{h}$ which are Lie algebra homomorphisms. The central extension is denoted by
		\[
		V \stackrel{\iota_{*}}{\to} \mathfrak{g} \stackrel{\pi_{*}}{\to} \mathfrak{h}.
		\]
		If $G$ and $H$ are simply connected Lie groups with Lie algebras $\mathfrak{g}$ and $\mathfrak{h}$, respectively, we also refer to the induced short exact sequence
		\[
		V \stackrel{\iota}{\to} G \stackrel{\pi}{\to} H
		\]
		as a central extension of $H$ by $\rho$.
	\end{definition}
	
	For the purpose of studying the functions that admit contact lift, we are supposed to define the central extension of Carnot groups. 
	
	\begin{definition}
		A central extension $V \xrightarrow{\iota} G \xrightarrow{\pi} H$ by $\rho$ is called \emph{stratified} if the following properties hold.
		\begin{enumerate}
			\item[(i)] The Lie algebras $\mathfrak{h}$ and $\mathfrak{g}$ are stratified and the vector space $V$ is graded.
			\item[(ii)] The inclusion $\iota_{*} \colon V \to \mathfrak{g}$ and the projection $\pi_{*} \colon \mathfrak{g} \to \mathfrak{h}$ are graded with respect to the gradings from (i).
		\end{enumerate}
		A stratified central extension $V \to G \xrightarrow{\iota} H$ by $\rho$ is a \emph{central extension of Carnot groups} if the following hold.
		\begin{enumerate}
			\item[(iii)] The Lie groups $G$ and $H$ are Carnot groups and $V$ is an inner product space.
			\item[(iv)] The inclusion $\iota_{*} \colon V \to \mathfrak{g}$ is an isometry and $\pi_{*} \colon \mathfrak{g} \to \mathfrak{h}$ is a submetry with respect to the inner product structures from (iii).
		\end{enumerate}
	\end{definition}
	
	\begin{remark}
		Consider Carnot groups $G$ and $H$ together with a graded homomorphism $\pi \colon G \to H$ for which $\pi_{*} \colon \mathfrak{g} \to \mathfrak{h}$ is a submetry and $\ker(\pi_{*})$ is contained in the center of $\mathfrak{g}$.
	\end{remark}
	\begin{definition}
		Let $X$ and $Y$ be metric spaces and $0 < \gamma \leq 1$. A mapping $f : X \to Y$ is $\gamma$-H\"older continuous if
		\[
		[f]_{\gamma} = [f]_{C^{0,\gamma}} := \sup\left\{\frac{d_{Y}(f(x),f(y))}{d_{X}(x,y)^{\gamma}} : x \neq y\right\} < \infty. 
		\]
		The space of all $\gamma$-H\"older continuous mappings $f : X \to Y$ will be denoted by $C^{0,\gamma}(X;Y)$.
		
		$C^{0,\gamma}_{b}(X;Y)$ is the space of bounded $\gamma$-H\"older continuous mappings i.e., the image of every mapping is contained in a ball in $Y$.
	\end{definition}
	
	Note that $C_{\mathrm{b}}^{0,\gamma}(X;\mathbb{R}^{m})$ is a Banach space with respect to the norm
	\[
	\|f\|_{C^{0,\gamma}} = \|f\|_{\infty} + [f]_{\gamma}.
	\]

	$C_{\mathrm{b}}^{0,\gamma}(X)$ is a real Banach space. Moreover if $f, g \in C_{\mathrm{b}}^{0,\gamma}(X)$, then $fg \in C_{\mathrm{b}}^{0,\gamma}(X)$ and $\|fg\|_{C^{0,\gamma}} \leq \|f\|_{C^{0,\gamma}} \|g\|_{C^{0,\gamma}}$. This algebra property will be used repeatedly in Section~\ref{sec-holder-lift}, where products of coordinate functions of H\"older maps have to be integrated against one another.
	
	With these preliminaries in hand, we can now turn to the lifting problem proper. Sections~\ref{sec-Lipschitz} and~\ref{sec-Sobolev} treat the two regularities for which a first-order calculus is available; the statements are parallel and so are the proofs, the difference being that the Sobolev case requires the curve family to be understood in the almost-everywhere sense.

	\section{Preservation of closed horizontal curves for Lipschitz contact lift}\label{sec-Lipschitz}
	In this section, we connect the existence of Lipschitz contact lifts to a condition on closed horizontal curves.
	
	The strategy is dictated by the smooth theory. In the smooth category the obstruction to lifting is a cohomology class, and cohomology classes are detected by periods; in the Lipschitz category one cannot integrate a distributional form over a submanifold, but one can still integrate a bounded $1$-form along a Lipschitz curve. So the whole problem reduces to understanding which closed horizontal curves acquire a non-trivial period, and the first step is to make the fibre coordinate of a horizontal curve explicit.
	
	For any horizontal curve, we can express its location in integral form on  central extension of Carnot groups.(cf. \cite{HHI2025}, Lemma 3.5])
	\begin{lem}\label{lem-curve-integral}
		Let $V \overset{\iota}{\to} G \overset{\pi}{\to} H$ be a central extension of Carnot groups by $\rho$ such that $rank(G) = rank(H)$. There exists a $1$-form $\alpha \in \Omega^1(H; V)$ such that
		\begin{enumerate}
			\item[(i)] $d\alpha = \rho$, and
			\item[(ii)] if $\gamma: [0,1] \to G$ is a horizontal curve starting from $\gamma(0) = e_G$, then
			\begin{equation*}
				log_G(\gamma(1)) = log_H(\pi \circ \gamma(1)) + \int_{\pi \circ \gamma} \alpha \in \mathfrak {h} \oplus V = \mathfrak{g} .
			\end{equation*}
		\end{enumerate}
	\end{lem}
	
	Lemma~\ref{lem-curve-integral} converts a lifting question into an arithmetic one: a horizontal curve in $H$ lifts to a horizontal curve in $G$ with prescribed endpoints precisely when the base endpoint and the period $\int\alpha$ match. In particular, a \emph{closed} horizontal curve in $H$ lifts to a \emph{closed} horizontal curve in $G$ if and only if its period vanishes. Since the map $f$ we want to lift is only Lipschitz, we must restrict attention to curves on which the line integral is unambiguous, and Lipschitz curves are exactly the class for which this is so.
	
	Hence, we can judge if a curve on Carnot group  $H$ can be lifted to a larger Carnot group $G$, and it is natural to record the answer in terms of the two curve families that will appear throughout the paper.
	Given a Carnot group $\mathbb{H}$, a point $h \in \mathbb{H}$, and a domain $U \subset \mathbb{H}$, we denote by $\Gamma_{\text{LIP}}(h, U)$ the collection of closed Lipschitz curves $\gamma: [0,1] \to U$ based at $\gamma(0) = h = \gamma(1)$.
	
	For a central extension $V \to \mathbb{G} \to \mathbb{H}$ of Carnot groups by a 2-cocycle $\rho \in \Omega^2(\mathbb{H}; V)$, we denote by $\Gamma_{\text{LIP}}^{\rho}(h, U) \subset \Gamma_{\text{LIP}}(h, U)$ the subcollection which admit a horizontal lift to $\mathbb{G}$ that is closed. Representing $\mathbb{G}$ as a direct product $\mathbb{G} \simeq \hat{\mathbb{G}} \times W$, where $W \subset V$ is the horizontal component of the central extension $V \to \mathbb{G} \to \mathbb{H}$, the subcollection can be characterized as
	\[
	\Gamma_{\text{LIP}}^{\rho}(h, U) = \left\{ \gamma \in \Gamma_{\text{LIP}}(h, U) : \int_{\gamma} \alpha = 0 \right\},
	\]
	where $\alpha \in \Omega^1(\mathbb{H}; V)$ satisfies $d\alpha = \rho$. Note that the collection is independent of the specific potential since $\int_{\gamma} d\psi = 0$ for any closed curve $\gamma \in \Gamma_{\text{LIP}}(h, U)$ and any exact 1-form $d\psi \in \Omega^1(\mathbb{H}; V)$. 
	
	In fact, we can use the lifting of horizontal curves on $H_2$ to identify if a map admits a contact lift. The point is that a contact lift $F$ must send horizontal curves to horizontal curves and must respect the projections, so $F\circ\gamma$ is forced to be a horizontal lift of $f\circ\pi_{1}\circ\gamma$; conversely, if every closed horizontal curve in $\mathbb{G}_{1}$ is sent to a curve in $\mathbb{H}_{2}$ whose lift closes up, then the lift can be defined by path lifting without ambiguity. The Lipschitz regularity of the resulting $F$ will follow from the fact that the fibre coordinate is controlled by the length of the base curve. We now prove the criterion announced as Lemma~\ref{lem-Lipschitz-intro} in the introduction.
	
	\begin{lem}\label{lem-Lipschitz-body}
		Let $U_1 \subset \mathbb{H}_1$ and $\tilde{U}_1 \subset \pi_1^{-1}(U_1) \subset \mathbb{G}_1$ be domains. A Lipschitz contact map $f: U_1 \to \mathbb{H}_2$ admits a Lipschitz contact lift $F: \tilde{U}_1 \to \mathbb{G}_2$ if and only if
		\begin{equation}\label{curve}
			(f \circ \pi_1)\left( \Gamma_{\text{LIP}}(g, \tilde{U}_1) \right) \subset \Gamma_{\text{LIP}}^{\rho}(f(\pi_1(g)), \mathbb{H}_2)
		\end{equation}
		for some $g \in \tilde{U}_1$.
	\end{lem}
	\begin{proof}
		If $F\colon\tilde{U}_{1}\to G_{2}$ is a Lipschitz contact lift of $f\colon U_{1}\to H_{2}$ and $\gamma\in\Gamma_{\text{LIP}}(g,\tilde{U}_{1})$ is a closed horizontal curve based at an arbitrary $g\in\tilde{U}_{1}$, then $F\circ\gamma$ is a closed horizontal curve in $G_{2}$ based at $F(g)$ whose projection to $H_{2}$ is the closed horizontal curve $\pi_{2}\circ F\circ\gamma=f\circ\pi_{1}\circ\gamma$ based at $f(\pi_{1}(g))$. Thus \eqref{curve} follows.
		
		For the converse direction, we first observe that by Lemma 4.1 in \cite{HHI2025}, it suffices to consider the case $\mathrm{rank}(G_{2})=\mathrm{rank}(H_{2})$. Suppose that \eqref{curve} holds for some $g\in\tilde{U}_{1}$. We will construct the lift $F$ by a standard path lifting argument seen. We refer the reader to Section 8 in
		\cite{ABB2019}  for background on the endpoint map.
		
		Fix a basepoint $p\in\pi_{2}^{-1}(f(\pi_{1}(g)))$. We construct a Lipschitz contact lift $F\colon\tilde{U}_{1}\to G_{2}$ of $f$ with $F(g)=p$. Recall that for every control $u\in L^{2}([0,1],\mathfrak{g}_{1}^{[1]})$, there exists a corresponding trajectory $\gamma_{u}\colon[0,1]\to G_{1}$ obtained as the unique solution of the Cauchy problem
		\begin{equation*}
			\dot{\gamma}_{u}(t)=(L_{\gamma_{u}(t)})_{*}u(t),\quad\gamma_{u}(0)=g.
		\end{equation*}
		This defines the endpoint map
		\begin{equation*}
			\text{End}_{g}\colon L^{2}([0,1];\mathfrak{g}_{1}^{[1]})\to G_{1},\quad\text{End}_{g}(u)=\gamma_{u}(1),
		\end{equation*}
		which by  Proposition 8.5 in \cite{ABB2019} is Lipschitz. We will also consider the two analogously defined endpoint maps
		\begin{equation*}
			\text{End}_{f(\pi_{1}(g))}\colon L^{2}([0,1];\mathfrak{h}_{2}^{[1]})\to H_{2},
		\end{equation*}
		\begin{equation*}
			\text{End}_{p}\colon L^{2}([0,1];\mathfrak{g}_{2}^{[1]})\to G_{2}.
		\end{equation*}
		We have $\mathfrak{h}_{2}^{[1]}=\mathfrak{g}_{2}^{[1]}$, so these latter two endpoint maps have the same domain. We note that $\text{End}_{f(\pi_{1}(g))}=\pi_{2}\circ\text{End}_{p}$ by construction.
		
		Let $W\subset L^{2}([0,1];\mathfrak{g}_{1}^{[1]})$ be the open subset of controls whose trajectories satisfy $\gamma_{u}([0,1])\subset\tilde{U}_{1}$. Since $\tilde{U}_{1}$ is a domain, $\text{End}_{p}(W)=\tilde{U}_{1}$. The Lipschitz contact map $f\circ\pi_{1}\colon\tilde{U}_{1}\to H_{2}$ induces a bounded map $(f\circ\pi_{1})_{*}\colon W\to L^{2}([0,1];\mathfrak{h}_{2}^{[1]})$ given by
		\begin{equation*}
			((f\circ\pi_{1})_{*}u)(t)=(L_{f\circ\pi_{1}\circ\gamma_{u}(t)}^{-1})_{*}\frac{d}{dt}(f\circ\pi_{1}\circ\gamma_{u}(t)).
		\end{equation*}
		We define $\Phi:=\text{End}_{p}\circ(f\circ\pi_{1})_{*}\colon W\to G_{2}$. Assumption \eqref{curve} implies that the value $\Phi(u)$ depends only on $\text{End}_{g}(u)$ for $u\in W$. Thus there exists a uniquely defined map $F\colon\tilde{U}_{1}\to G_{2}$ for which $F\circ\text{End}_{g}\vert_{W}=\Phi$.
		
		By a perturbation argument, we see that, for any $h\in\tilde{U}_{1}$, there exists a control $u\in W\cap\text{End}_{g}^{-1}(h)$ such that the differential $(d\text{End}_{g})_{u}$ has full rank. So, by the implicit function theorem, every $h\in\tilde{U}_{1}$ has an open neighborhood $W \subset \tilde{U}_{1}$ such that a smooth right inverse $\Psi \colon V \to W$ of $\mathrm{End}_g$ exists. Therefore we have $\left.F\right|_{W} = \Phi \circ \Psi$. Moreover, $\pi_{2} \circ F = f \circ \pi_{1}$. Indeed, if $\gamma \colon [0,1] \to \tilde{U}_{1}$ is a horizontal curve, then $F \circ \gamma$ is a horizontal lift of $f \circ \pi_{1} \circ \gamma$ by the construction, the equality $\mathrm{End}_f(\pi_1(g)) = \pi_2 \circ \mathrm{End}_p$, and the assumption , and thus the lifting property follows. It also follows that $F$ is contact. 
		
		If $\gamma:[0,1]\to H_1$ is a horizontal curve starting from $\gamma(0)=e_{H_1}$, then 
		\begin{equation*}
			\widetilde{\gamma}(1)=\gamma(1) \cdot exp(\int_{\gamma}\alpha)
		\end{equation*}
		For the contact lift $F$ of $f$, then
		\begin{equation*}
			F(\widetilde{\gamma}(1))=f(\gamma(1)) \cdot exp(\int_{\gamma}\left(f \circ \pi_{1}\right)^{*} \alpha_{2})
		\end{equation*}
		By the assumption, we know $\alpha_2\in\Omega^1(H_2;V_2)$, so locally there is a constant $C$ such that $\left||\alpha_2|\right|\leq C$, then $\left||(f\circ \pi_1)_*\alpha_2|\right|\leq LC$ , where $L$ is the Lipschitz constant of $f$.
		Thus, $F$ is Lipschitz contact lift of $f$.
		
	\end{proof}
	
	\section{Preservation of closed horizontal curves for Sobolev contact lift}\label{sec-Sobolev}
	In this section, we connect the existence of Sobolev contact lifts to a condition on closed horizontal curves.
	
	The formulation of the Lipschitz criterion in Lemma~\ref{lem-Lipschitz-body} is sufficiently robust that one expects the same statement in the Sobolev category, and this is indeed the case. Two new issues arise. First, a Sobolev map need not be defined pointwise, so the closed-curve condition can only be required along almost every curve; this is where the homogeneous dimension $Q_{1}$ enters, through the modulus of the family of curves on which the map is badly behaved. Second, and this is the substantive point, the converse direction no longer produces a Lipschitz map: one has to reconstruct the Sobolev regularity of $F$ from an upper gradient, rather than from a pointwise Lipschitz estimate.
	
	Throughout we use the upper-gradient characterisation of Sobolev maps on metric measure spaces, for which we refer to \cite{HKST2015}; in particular, a map belongs to $W^{1,p}$ as soon as it admits a $p$-integrable upper gradient.
	We denote by $\Gamma_{W^{1,p}}(h, U)$ the collection of closed $W^{1,p}$ admissible curves $\gamma: [0,1] \to U$ based at $\gamma(0) = h = \gamma(1)$.
	
	We denote by $\Gamma_{W^{1,p}}^{\rho}(h, U) \subset \Gamma_{W^{1,p}}(h, U)$ the subcollection which admit a horizontal lift to $\mathbb{G}$ that is closed. The subcollection can be characterized as
	\begin{equation*}
		\Gamma_{W^{1,p}}^{\rho}(h, U) = \left\{ \gamma \in \Gamma_{W^{1,p}}(h, U) : \int_{\gamma} \alpha = 0 \right\}.
	\end{equation*}
	
	So, we can draw a similar conclusion as Lipschitz version just for Sobolev contact lift, at the cost of replacing ``every curve'' by ``almost every curve'' with respect to the relevant Hausdorff measure. This is the Sobolev counterpart of Lemma~\ref{lem-Lipschitz-body}, and it is the criterion announced as Lemma~\ref{lem-Sobolev-intro} in the introduction.
	\begin{lem}\label{lem-Sobolev-body}
		Let $U_1 \subset \mathbb{H}_1$ and $\tilde{U}_1 \subset \pi_1^{-1}(U_1) \subset \mathbb{G}_1$ be domains. A Sobolev contact map $f: U_1 \to \mathbb{H}_2$ admits a Sobolev contact lift $F: \tilde{U}_1 \to \mathbb{G}_2$ if and only if
		\begin{equation}\label{curve-S}
			(f \circ \pi_1)\left( \Gamma_{\text{LIP}}(g, \tilde{U}_1) \right) \subset \Gamma_{W^{1,p}}^{\rho}(f(\pi_1(g)), \mathbb{H}_2) \ \text{for}\  \mathcal{H}^{Q_1-1}-a.e.
		\end{equation}
		for some $g \in \tilde{U}_1$, where $f,F\in W^{1,p}$ and $p>Q_1$, $Q_1$ is homogeneous dimension of $H_1 $.
	\end{lem}
	\begin{proof}
		For any $\gamma \in \Gamma_{\text{LIP}}(g, \tilde{U}_1)$, because $F\in W^{1,p}$ and $D_pF$ is a Lie algebra homomorphism , $F\circ \gamma$ is a horizontal curve. By $D_pF \in L^p (H_1)$, we can take a tubular neighbor $U_r$ of $\gamma$ with radius r , then
		\begin{equation*}
			\int_{U_r}|D_p (f)|^p < \infty
		\end{equation*} 
		$f$ is differential on almost every curve and $f\circ \pi_1\circ$$\gamma \in \Gamma_{W^{1,p}}$ for $\mathcal{H}^{Q_1-1}-a.e$. Therefore we have $\int_{\gamma}\left(f \circ \pi_{1}\right)^{*} \alpha_{2}$ exists $\mathcal{H}^{Q_1-1}-a.e$. Due to $F$ is contact lift $\pi_2\circ F \circ \gamma=f \circ \pi_1\circ \gamma$, we have $\int_{\gamma}\left(f \circ \pi_{1}\right)^{*} \alpha_{2}=0$ for $\mathcal{H}^{Q_1-1}-a.e$. Thus \eqref{curve-S} follows.
		
		For the converse direction, we first observe that $f\in \Gamma_{W^{1,p}}$ , $p>Q_1$, so we can take $f$ as H\"older continuous. And also because $L^p([0,1])\subset L^2([0,1])$, we can find that the conditions for endpoint maps are all satisfied for $f \circ \pi_1\circ \gamma$ $\mathcal{H}^{Q_1-1}-a.e$. Therefore we can use the same method as the proof of Lipschitz contact lift to get a lift $F$ of $f$. Obviously, $F$ is continuous, since $\int_{\gamma}\left(f \circ \pi_{1}\right)^{*} \alpha_{2}=0$ for $\mathcal{H}^{Q_1-1}-a.e$, we can take a Lipschitz $\widetilde{\gamma}$ to define $F$ without worrying about the existence of the limit of the integral
		\begin{equation*}
			F(\widetilde{\gamma}(1))=f(\gamma(1)) \cdot exp(\int_{\widetilde{\gamma}}\left(f \circ \pi_{1}\right)^{*} \alpha_{2}),
		\end{equation*}
		where $\widetilde{\gamma}(0)=e_G$.
		
		The last step is to prove $F\in W^{1,p}$. We can assume $f(e)=e$, so $F(e)=e$. For almost every Lipschitz curve $\widetilde{\gamma}$, $\widetilde{\gamma}(0)=e$
		\begin{equation*}
			\begin{aligned}
				&d_{CC}(F(\widetilde{\gamma}(1)),e)\\ 
				&\leq d_{CC}(e,exp(\int_{\widetilde{\gamma}}\left(f \circ \pi_{1}\right)^{*} \alpha_{2}))\\
				&+d_{CC}(exp(\int_{\widetilde{\gamma}}\left(f \circ \pi_{1}\right)^{*} \alpha_{2}),f(\gamma(1)) \cdot exp(\int_{\widetilde{\gamma}}\left(f \circ \pi_{1}\right)^{*} \alpha_{2}))\\
				&= d_{CC}(e,exp(\int_{\widetilde{\gamma}}\left(f \circ \pi_{1}\right)^{*} \alpha_{2}))+d_{CC}(e,f(\gamma(1))\\
				&\leq \int_{\widetilde{\gamma}}|\left(f \circ \pi_{1}\right)^{*} \alpha_{2}| + \int_{\widetilde{\gamma}}|\nabla_H f|ds\\
				&=\int_{\widetilde{\gamma}}(|\left(f \circ \pi_{1}\right)^{*} \alpha_{2}|+|\nabla_H f|)ds,
			\end{aligned}
		\end{equation*}
		so $|\left(f \circ \pi_{1}\right)^{*} \alpha_{2}|+|\nabla_H f|$ is the upper gradient of $F$ and it belongs to $L^p$. Therefore $F\in W^{1,p}(G_1;G_2)$
	\end{proof}
	
	Having settled the existence question, we now turn to the structure of the lifts that do exist, and this is what will ultimately produce the rigidity theorems of Section~\ref{sec-rigidity}. The key structural fact is that a lift is determined by its values on the base together with a homomorphism between the fibres; equivalently, translating a point along the central fibre $V_{1}$ translates its image along $V_{2}$ by an amount independent of the base point. If we want to study the contact lift of quasiconformal maps we still need to know the fiber component of a sobolev lift and the decomposition of a Lipschitz or sobolev contact lift.
	
	\begin{lem}\label{lem-fibre-decomp}
		Let $V_1 \to G_1 \to H_1$ and $V_2 \to G_2 \to H_2$ be central extensions of Carnot groups such that $\mathrm{rank}(G_2) = \mathrm{rank}(H_2)$. Let $U_1 \subset H_1$ be a domain and $F: \pi_1^{-1}(U_1) \to G_2$ a Sobolev contact lift of a sobolev contact map $f: U_1 \to H_2$. Then there exists a Lie group homomorphism $\Phi: V_1 \to V_2$ such that $F(gk) = F(g)\Phi(k)$ for all $g \in \pi_1^{-1}(U_1)$ and $k \in V_1$.
	\end{lem}
	\begin{proof}
		By the lift assumption, we have
		\[
		\pi_2 \circ F(gk) = f \circ \pi_1(gk) = f \circ \pi_1(g) = \pi_2 \circ F(g).
		\]
		Thus, for each $g \in \pi_1^{-1}(U_1)$ and $k \in V_1$, there exists some $\tilde{\Phi}(g,k) \in V_2$ such that $F(gk) = F(g)\tilde{\Phi}(g,k)$. We claim that $\tilde{\Phi}(g,k)$ is constant in $g \in \pi_1^{-1}(U_1)$. That is, there is a well defined map $\Phi: V_1 \to V_2$ such that $\tilde{\Phi}(g,k) = \Phi(k)$ for $g \in \pi_1^{-1}(U_1)$ and $k \in V_1$. By assumption, $U_1$ is connected, so it suffices to prove that $g \mapsto \tilde{\Phi}(g,k)$ is locally constant for fixed $k \in V_1$.
		
		Let $B'$ be a ball that is compactly contained in $\pi_1^{-1}(U_1)$. Let $B$ be a ball with the same center and half the radius.
		
		Let $g, h \in B$. Since the extensions are central, elements of $V_1$ commute with elements of $G_1$, and similarly for $V_2$. By left-invariance of the distance, we compute that
		\[
		\begin{aligned}
			d(\tilde{\Phi}(g,k), \tilde{\Phi}(h,k)) &= d(F(g)\tilde{\Phi}(g,k), F(g)\tilde{\Phi}(h,k)) \\
			&\leq d(F(g)\tilde{\Phi}(g,k), F(h)\tilde{\Phi}(h,k)) \\
			&\quad + d(F(h)\tilde{\Phi}(h,k), F(g)\tilde{\Phi}(h,k)) \\
			&= d(F(gk), F(hk)) + d(F(h), F(g)) \\
		\end{aligned}
		\]
		Because $F\in W^{1,p}$ and $|\nabla_H F|\in L^p$ is the upper gradient of $F$. Taking a geodesic curve connecting $h$ and $g$, then $\gamma\cdot k$ is a geodesic curve connecting $hk$ and $gk$. Therefore, we have
		\begin{equation*}
			\begin{aligned}
				&d(F(h), F(g))+d(F(gk), F(hk))   \\
				&\leq \int_{\gamma}|\nabla_H F|+\int_{\gamma\cdot k}|\nabla_H F|\\
				&\leq \int_{\gamma}|\nabla_H F|+\int_{\gamma}|\nabla_H F|\circ r_k,
			\end{aligned}
		\end{equation*}
		where $r_k$ is the right multiplication by k. Therefore, the $ \int_{\gamma}|\nabla_H F|+\int_{\gamma}|\nabla_H F|\circ r_k$ is the upper gradient of $\tilde{\Phi}(g,k)$, which means $\tilde{\Phi}(g,k)\in W^{1,p}$.
		By the equal rank assumption, $V_2$ does not contain any horizontal directions. Hence the only possible sobolev maps $B \to V_2$ are constants. Consequently, $g \mapsto \tilde{\Phi}(g,k)$ is locally constant for $k \in V_1$.
		
		Next, we show that $\Phi: V_1 \to V_2$ is a group homomorphism. Let $k_1, k_2 \in V_1$ and let $g \in \pi_1^{-1}(U_1)$. Applying the definition of $\tilde{\Phi}$ with the pair $(g, k_1 k_2)$ shows
		\[
		F(g k_1 k_2) = F(g) \Phi(k_1 k_2),
		\]
		whereas applying the definition with the pairs $(g k_1, k_2)$ and $(g, k_1)$ yields
		\[
		F(g k_1 k_2) = F(g k_1) \Phi(k_2) = F(g) \Phi(k_1) \Phi(k_2).
		\]
		Canceling out $F(g)$ implies that $\Phi$ is a group homomorphism.
		
		Finally, we observe that $\Phi(k) = F(g)^{-1}F(gk)$ for any $g \in G_{1}$ and $k \in V_{1}$, so $\Phi$ is a $W^{1,p}$ map. 
	\end{proof}
	
	The homomorphism $\Phi$ of Lemma~\ref{lem-fibre-decomp} is only the fibre component of a much more rigid object. If we want to study the Pansu differential of a contact lift, the first thing is to study the Lie algebra homomorphism between central extension sequences: the Pansu differential of a lift, being a graded Lie algebra homomorphism, must itself fit into a morphism of central extensions, and the next lemma describes exactly which fibre maps $\varphi$ can occur.
	
	\begin{lem}\label{lem-Pansu-diff}
		Let $V_1 \to \mathfrak{g}_1 \to \mathfrak{h}_1$ and $V_2 \to \mathfrak{g}_2 \to \mathfrak{h}_2$ be central extensions by $2$-cocycles $\rho_1$ and $\rho_2$, respectively. Let $L: \mathfrak{h}_1 \to \mathfrak{h}_2$ be a Lie algebra homomorphism and let $\varphi: V_1 \to V_2$ be a linear map.
		
		(i) There exists a Lie algebra homomorphism $\psi: \mathfrak{g}_1 \to \mathfrak{g}_2$ such that the diagram
		\begin{displaymath}
			\xymatrix{
				V_1 \ar[r] \ar[d]_\varphi & \mathfrak{g}_1 \ar[d]^{\psi} \ar[r]^{\pi_1} & \mathfrak{h}_1 \ar[d]^{L}
				\\
				V_2 \ar[r] & \mathfrak{g}_2  \ar[r]^{\pi_2} & \mathfrak{h}_2
			}
		\end{displaymath}
		
		commutes if and only if there exists a linear map $\mu: \mathfrak{h}_1 \to V_2$ such that $\varphi \circ \rho_1 - L^* \rho_2 = d_0 \mu$. When such a $\mu$ exists, a homomorphism $\psi$ is given by
		\begin{equation}
			\psi(X + Y) = L(X) + (\mu(X) + \varphi(Y)) \in \mathfrak{h}_2 \oplus V_2 = \mathfrak{g}_2
		\end{equation}
		for $X + Y \in \mathfrak{h}_1 \oplus V_1 = \mathfrak{g}_1$.
		
		(ii) When the central extensions are stratified and $\varphi$ and $L$ are graded, a graded homomorphism $\psi$ is unique up to a graded linear map $\theta: \mathfrak{h}_1 \to V_2$ with $[\mathfrak{h}_1, \mathfrak{h}_1] \subset \ker \theta$. In particular, if $\mathfrak{h}_2^{[1]} = \mathfrak{g}_2^{[1]}$, then the graded homomorphism $\psi$ is unique.
		
	\end{lem}
	
	Because $\Phi$ is Lie group homomorphism and $V_1$,$V_2$ are central, so $\Phi$ is a linear map.
	\begin{equation*}
		\Phi(2k)=\Phi(k+k)=\Phi(k)+\Phi(k)=2\Phi(k)
	\end{equation*}
	Let $g \in \pi_{1}^{-1}(U_{1})$. Computing the Pansu derivative in a direction $h \in V_{1}$, we obtain
	\[
	\begin{aligned}
		d_{P} F(g)h &= \lim_{\lambda \rightarrow 0^{+}} \delta_{1 / \lambda}\left(F(g)^{-1} F\left(g \delta_{\lambda} h\right)\right) \\
		&= \lim_{\lambda \rightarrow 0^{+}} \delta_{1 / \lambda}\left(F(g)^{-1} F(g) \Phi\left(\delta_{\lambda} h\right)\right) \\
		&= \lim_{\lambda \rightarrow 0^{+}} \delta_{1 / \lambda}\left(\Phi\left(\delta_{\lambda} h\right)\right) =: \varphi h
	\end{aligned}
	\]
	
	Since the projections $\pi_{1}$ and $\pi_{2}$ are homomorphisms, we have $\pi_{2} \circ d_{P} F(g) = d_{P} f\left(\pi_{1} g\right) \circ \pi_{1}$. This gives us the commutative diagram:
	\begin{displaymath}
		\xymatrix{
			V_1 \ar[r] \ar[d]_\varphi & G_1 \ar[d]^{d_P F(g)} \ar[r]^{\pi_1} & H_1 \ar[d]^{d_P f(\pi_1g)}
			\\
			V_2 \ar[r] & G_2  \ar[r]^{\pi_2} & H_2
		}
	\end{displaymath}
	By Lemma \ref{lem-Pansu-diff}, such a diagram implies that the corresponding Lie algebra homomorphisms satisfy
	\begin{equation*}
		\varphi \circ \rho_{1} - d_{P} f(h)^{*} \rho_{2} = d_{0} \mu_{h}
	\end{equation*}
	for some linear map $\mu_{h} \colon \mathfrak{h}_{1} \to V_{2}$. Thus, $\mu_{h}=\varphi \circ \alpha_{1} - d_{P} f(h)^{*} \alpha_{2}$ and 
	\begin{equation}\label{Pansu-diff}
		D_P F=
		\begin{bmatrix}
			D_P f & 0 \\
			\varphi \circ \alpha_{1} - d_{P} f(h)^{*} \alpha_{2} & \varphi
		\end{bmatrix}
	\end{equation}
	With the result above, we know the Pansu differential for the lift of a quasiconformal contact map. In the block form \eqref{Pansu-diff}, the restriction of $D_{P}F$ to the horizontal layer coincides with the restriction of $D_{P}f$, because the added fibre directions are central and carry no horizontal component. It is therefore reasonable to expect that a quasiconformal map whose lift exists forces the lift to be quasiconformal as well, and get the lift is also a quasiconformal contact map when $V_2$ only has trivial horizontal part.
	
	\begin{lem}\label{lem-quasi}
		Let $U_1 \subset \mathbb{H}_1$ and $\tilde{U}_1 \subset \pi^{-1}(U_1) \subset \mathbb{G}_1$ be domains. When $rank(\mathbb{G}_2)=rank(\mathbb{H}_2)$, a quasiconformal contact map $f: U_1 \to \mathbb{H}_2$ admits a quasiconformal contact lift $F: \tilde{U}_1 \to \mathbb{G}_2$ if and only if this quasiconformal contact map $f: U_1 \to \mathbb{H}_2$ admits a sobolev contact lift $F: \tilde{U}_1 \to \mathbb{G}_2$.
	\end{lem}
	\begin{proof}
		We know that a quasiconformal contact map is a $W^{1,Q_1}$ map. Thus, we just need to prove the restriction of $D_P F$ to the horizontal layer is related to a quasiconformal map between Euclidean Spaces. Because of \eqref{Pansu-diff}, the restriction of $D_P F$ to the horizontal layer only depends on the restriction of $D_P f$ to the horizontal layer. So, the sobolev contact lift $F$ is also quasiconformal contact lift.
	\end{proof}
	We conclude this section by recording the smooth criterion of \cite{HHI2025} which motivated the whole approach. It says that, in the presence of enough smoothness, the closed-curve condition of Lemma~\ref{lem-Lipschitz-body} can be replaced by the purely cohomological requirement that the two cocycles define the same class in the Rumin complex; the equivalence between the two formulations is exactly what the chain-map property of the Pansu pullback provides.
	
	\begin{thm}\label{thm-smooth-criterion}
		Suppose that $f_{P}^{*}\rho_{2} = \varphi \circ \rho_{1} + d_{0}\omega$ for a graded linear map $\varphi \colon V_{1} \to V_{2}$ and a 1-form $\omega \in \Omega^{1}(U_{1}; V_{2})$. If the weight of $\rho_{2}$ is equal to or greater than the maximal weight of non-trivial 2-cocycles in the Rumin complex $E_{0}$ of $\mathbb{H}_{1}$ and $U_{1}$ is simply connected, then $f$ admits a smooth contact lift $F \colon \mathbb{G}_{1} \supset \pi_{1}^{-1}(U_{1}) \to \mathbb{G}_{2}$.
	\end{thm}
	
	Theorem~\ref{thm-smooth-criterion} is easiest to appreciate in the model case, where the Rumin complex of the base is small enough that the cohomological condition collapses to a single scalar equation.
	
	\noindent\textbf{Example~\ref{ex-heisenberg}.}\label{ex-heisenberg} For Heisenberg group $H_1$
	\begin{displaymath}
		\xymatrix{
			R \ar[r] & H_1 \ar[d]^{F} \ar[r]^{\pi} & \mathbb{R}^2 \ar[d]^{f}
			\\
			R \ar[r] & H_1  \ar[r]^{\pi} & \mathbb{R}^2
		}
	\end{displaymath}
	We know that for Heisenberg group $H_1$ if $\varphi \circ \alpha_{1} - d_{P} f(h)^{*} \alpha_{2}=0$ then there is a contact lift $F$ of $f$. That is $\varphi=det(Df)$ is a constant. For instance, $f(a,b)=(a+b^2,b)$ admits a smooth contact lift; note that the base map is not affine, so the rigidity encoded in the cocycle condition concerns the Jacobian determinant rather than the map itself.
	\begin{equation*}
		Df=
		\begin{bmatrix}
			1 & 2b \\
			0 & 1
		\end{bmatrix}
	\end{equation*}
	
	\section{Rigidity of contact lift}\label{sec-rigidity}
	In this chapter we find that all of the quasiconformal contact maps which admit contact lifts are bi-Lipschitz by the rigidity of the lift. Especially, for Carnot groups with rigid structure, the possible maps are very few.
	
	The heuristic is as follows. A quasiconformal map distorts the horizontal layer by a bounded factor, but the lift condition couples the horizontal distortion to the behaviour of the fibre coordinate; because the fibre sits in a higher layer, it scales like $r^{i}$ when the horizontal layer scales like $r$, and this mismatch is only compatible with quasiconformality if the distortion is in fact bounded above and below. In other words, the existence of a lift upgrades ``bounded distortion'' to ``bounded and bounded-away-from-zero distortion'', i.e. to bi-Lipschitzness, which is the content of Theorem~\ref{thm-biLipschitz-intro}. The proof below makes this precise by computing the determinant of the Pansu differential in terms of the single scalar governing the fibre.
	
	\begin{thm}\label{thm-bi-Lips}
		Suppose that $f$ is a $K$-quasiconformal contact map and it admits a contact lift $F$ then $f$ is bi-Lipschitz.
	\end{thm}
	\begin{proof}
		\begin{displaymath}
			\xymatrix{
				V_1 \ar[r] \ar[d]_\varphi & G_1 \ar[d]^{F} \ar[r]^{\pi} & H_1 \ar[d]^{f}
				\\
				V_1 \ar[r] & G_1  \ar[r]^{\pi} & H_1
			}
		\end{displaymath}
		
		First of all, since Lemma \ref{lem-quasi} $F$ is a quasiconformal map so that the dimensions of image of $F$ must be the same as $G_1$ so as to $\varphi$ is non-degenerate.
		
		Let's begin with the easy condition when the central extension is for the highest layer, so for Case 1,2,3 we can directly compute by the restriction of $\varphi$. 
		
		Case 1: $V_1=\mathbb{R}$ and $\mathfrak{h}_1=\mathbb{R}^k$, so the central extension $G_1$ is a step 2 Carnot group. we know $V_1=span\left\{\alpha\right\}$ and
		\begin{equation*}
			\alpha=\sum_{i<j}a_{ij}dx_i\wedge dx_j
		\end{equation*}
		After using the following method below,
		\begin{equation*}
			\begin{aligned}
				&dx_1\wedge dx_2+dx_1\wedge dx_3+dx_2\wedge dx_3\\
				&=dx_1\wedge(dx_2+dx_3)+ dx_2\wedge(dx_2+dx_3)\\
				&=(dx_1+dx_2)\wedge(dx_2+dx_3)
			\end{aligned}
		\end{equation*}
		we can find a new basis of $\alpha$ like $\left\{dx_1+dx_2,dx_2+dx_3\right\}$ and note that as $\left\{d\widetilde{x}_1,d\widetilde{y}_2\right\}$.
		Thus, the $\alpha$ can be written as the form
		\begin{equation*}
			\alpha=\sum_{i}b_i d\widetilde{x}_i\wedge d\widetilde{y}_i.
		\end{equation*}
		Because $V_1$ is the second layer, $\varphi\circ \alpha=\lambda \alpha$ and ${f_P}^*\alpha=\lambda \alpha$.
		\begin{equation*}
			\sum_{i}b_i {f_P}^*(d\widetilde{x}_i\wedge d\widetilde{y}_i)=\sum_{i}b_i \lambda(d\widetilde{x}_i\wedge d\widetilde{y}_i)
		\end{equation*}
		Restricting the Pansu differential to  horizontal layer, we can get $\lVert D_Hf\rVert\approx |\lambda|^\frac{1}{2}$.
		
		From the assumption of $K$-$QC$ map $f$, we know $f$ is bi-Lipschitz, since
		\begin{equation*}
			|\lambda|^\frac{1}{2} K^{-1}\lVert X\rVert \leq |D_Hf(X)|\leq |\lambda|^\frac{1}{2} K \lVert X\rVert 
		\end{equation*}
		for horizontal vector $X$.
		
		Case 2: If $V_1={V_1}^{[p+1]}=span\left\{\alpha\right\}$  and $\mathfrak{h}_1={\mathfrak{h}_1}^{[1]}\oplus \dots \oplus  {\mathfrak{h}_1}^{[p]}$ ,for $p\geq 2$. 
		
		Due to $d_0\rho=0$, 
		\begin{equation*}
			\rho([X_1,X_2],X_3)+\rho([X_2,X_3],X_1)+\rho([X_3,X_1],X_2)=0
		\end{equation*}
		\begin{equation*}
			\rho({\mathfrak{h}_1}^{[i]},{\mathfrak{h}_1}^{[j]})\subset \rho({\mathfrak{h}_1}^{[i-1]},{\mathfrak{h}_1}^{[j+1]}) \subset\dots \subset\rho({\mathfrak{h}_1}^{[1]},{\mathfrak{h}_1}^{[i+j-1]})
		\end{equation*}
		$\alpha$ has the form 
		\begin{equation*}
			\alpha=\sum_{i}b_id\widetilde{x}_i\wedge d\widetilde{y}_i
		\end{equation*}
		where $d\widetilde{x}_i\in ({H_1}^{[1]})^*$ and $d\widetilde{y}_i\in ({H_1}^{[p]})^*$ with the weight $wt(d\widetilde{x}_i)=-1$ and $wt(d\widetilde{y}_i)=-p$.
		
		Because $V_1$ is the highest layer, $\varphi\circ \alpha=\lambda \alpha$ and ${f_P}^*\alpha=\lambda \alpha$ so that we can get $\lVert D_Hf\rVert\approx |\lambda|^\frac{1}{p+1}$.
		
		Case 3:If $V_1={V_1}^{[p+1]}=span\left\{\alpha_1\dots \alpha_j\right\}$  and $\mathfrak{h}_1={\mathfrak{h}_1}^{[1]}\oplus \dots \oplus  {\mathfrak{h}_1}^{[p]}$ for $p\geq 2$. 
		
		So the same as Case 2, we can get  ${f_P}^*\alpha=\varphi\circ \alpha$.
		When $D_Pf$ is non-degenerate on ${\mathfrak{h}_1}^{[1]}$, $f_{P*}({\mathfrak{h}_1}^{[1]})={\mathfrak{h}_1}^{[1]}$. Using all the vectors on ${\mathfrak{h}_1}^{[1]}$, they generate the whole Lie algebra $\mathfrak{g}_1$, so  $f_{P*}({\mathfrak{g}_1}^{[1]})={\mathfrak{g}_1}^{[1]}$.
		
		If $\Vert f_{P*}(X)\Vert \approx r\Vert X\Vert$ for horizontal vector $X$, 
		\begin{equation*}
			\begin{aligned}
				\Vert f_{P*}([X_1,X_2])\Vert &=\Vert [f_{P*}(X_1),f_{P*}(X_2)]\Vert\\
				&\leq 2\Vert f_{P*}(X_1)\Vert \cdot \Vert f_{P*}(X_2)\Vert \\
				&\lesssim2 r^2\Vert X_1\Vert \cdot\Vert X_2\Vert
			\end{aligned}
		\end{equation*}
		where $X_1$ and $X_2$ are horizontal vectors. And also because $f_{P*}$ is non-degenerate, there is a lower bound of $f_{P*}([\cdot,\cdot])$. 
		combined with $D_Pf$ commuting with dilation $\delta_r$, therefore, if the horizontal layer is expanded by $r$ times, then the k-th layer is expanded by $r^k$ times. If $\Vert \varphi \Vert \approx \lambda$, then  $\lVert D_Hf\rVert\approx |\lambda|^\frac{1}{p+1}$.
		
		Case 4: the central extension is not to add the highest layer, so the structure of $D_PF$ of the lifting $F$ is not good enough. We need to estimate the determinant of $ D_Hf$ in a more general way which depends on the existence of the quasiconformal contact lift.
		
		Because of Lemma \ref{lem-quasi}, $f$ admits a quasiconformal contact lift $F$ and
		\begin{equation*}
			D_PF|_{{\mathfrak{g}_1}^{[1]}}=\varphi|_{{\mathfrak{g}_1}^{[1]}}\subset {\mathfrak{g}_1}^{[1]}.
		\end{equation*}
		The homogeneous dimension of ${\mathfrak{g}_1}^{[1]}$ is $Q'$. If $det(\varphi)=\lambda$, then $\Vert D_Hf \Vert\approx \lambda^\frac{1}{Q'}$.
		
	\end{proof}
	
	The result is not finished here, for some Carnot groups with rigid structure, we can even draw a better conclusion. The reason is that on a quaternionic Heisenberg group the central fibre is three-dimensional and carries three anticommuting complex structures, so the lift condition does not merely bound the distortion, it pins down the differential completely. This is the same phenomenon behind the Liouville-type theorems of Capogna--Cowling \cite{CapognaCowling2006} and Ottazzi--Warhurst \cite{OttazziWarhurst2011}, and behind the affine rigidity of isometries proved by Le Donne--Ottazzi \cite{LeDonneOttazzi2016}: once a contact map preserves sufficiently many algebraic structures, it belongs to a finite-dimensional group. The following is Theorem~\ref{thm-quaternionic-intro} of the introduction. For example, if $F$ is a contact lift on quaternionic Heisenberg group $\mathbb{H}_\mathbb{H}^n$ of a map $f$ on $\mathbb{R}^{4n}$, then $f:\mathbb{R}^{4n}\rightarrow\mathbb{R}^{4n}$ is a affine transformation.
	
	\begin{thm}\label{thm-R4n}
		Suppose $f$ is a Sobolev continuous map on $\mathbb{R}^{4n}$ which admits a contact lift $F$ on $\mathbb{H}_\mathbb{H}^n$ and $D_PF$ is full-rank for every point, then $f$ is an affine transformation.
	\end{thm}
	\begin{proof}
		By assumption, $F$ is a contact lift and its Pansu differential is 
		\begin{equation*}
			D_PF=
			\begin{bmatrix}
				Df\\
				&\varphi
			\end{bmatrix}
		\end{equation*}
		for some constant matrix $\varphi$. 
		\begin{equation*}
			\varphi=
			\begin{bmatrix}
				a_{11}& a_{12} &a_{13}\\
				a_{21}& a_{22} &a_{23}\\
				a_{31}& a_{32} &a_{33}
			\end{bmatrix}
		\end{equation*}
		
		We define central extension as $\rho = (\rho_{1},\rho_{2},\rho_3)$,
		\begin{equation*}
			\begin{aligned}
				&\rho_1=\sum_{i=1}^{n} dx_i\wedge dy_i+dz_i \wedge dw_i\\
				&\rho_2=\sum_{i=1}^{n} dx_i\wedge dz_i+dw_i \wedge dy_i\\
				&\rho_3=\sum_{i=1}^{n} dx_i\wedge dw_i+dy_i \wedge dz_i.
			\end{aligned}
		\end{equation*}  
		
		For any fixed point $q$, $D_PF(q)$ is a group homomorphism , because $D_PF$ is invertible for every point, we can define a map $\widetilde{F}=F\circ {D_PF(q)}^{-1}$ with Pansu differential 
		\begin{equation}
			D_P\widetilde{F}=
			\begin{bmatrix}
				Df\circ (Df(q))^{-1}\\
				&id
			\end{bmatrix}
		\end{equation} 
		Obviously, $\widetilde{F}$ is also a contact lift, since $\widetilde{f}^*(\alpha)=id \circ \alpha$ for constant map $id$, where $\widetilde{f}=f\circ (Df(q))^{-1}$. 
		Thus, $\widetilde{F}_P^*(\rho)=\widetilde{f}^*(\rho)=\rho$, which means $\widetilde{f}$ is a symplectic forms preserving map. This would greatly limit the possible choice of $f$.
		
		The symplectic forms $\rho_1,\rho_2,\rho_3$ can be represent as the matrix $J_1,J_2,J_3$ in local coordinates.
		\begin{equation*}
			J_1=\begin{bmatrix}  
				& I_n &  & \\
				-I_n&  &  & \\
				&  &  &I_n \\
				&  &  -I_n&
			\end{bmatrix}
		\end{equation*} 
		\begin{equation*}
			J_2=\begin{bmatrix}  
				&  &  I_n& \\
				&  &  &I_n \\
				-I_n&  &  & \\
				&  -I_n&  &
			\end{bmatrix}
		\end{equation*}
		\begin{equation*}
			J_3=\begin{bmatrix}  
				&  &  &-I_n \\
				&  &-I_n  & \\
				&  I_n&  & \\
				I_n&  &  &
			\end{bmatrix}
		\end{equation*}
		
		where satisfying $J_3=J_1J_2$. We denote $D\widetilde{f}:=A$. Due to map $\widetilde{f}$ preserving symplectic forms, we have
		\begin{equation*}
			A^TJ_iA=J_i \ \ for \ i=1,2,3.
		\end{equation*}
		and 
		\begin{equation*}
			\begin{aligned}
				A^TJ_3A&=J_3\\
				A^T(J_1J_2)A&=J_1J_2\\
				A^T(J_1J_2)A&=A^TJ_1AA^TJ_2A\\
				J_1J_2&=J_1AA^TJ_2\\
				I&=AA^T,
			\end{aligned}
		\end{equation*}
		since $A$ is invertible. Thus, $A$ is orthogonal.
		\begin{equation*}
			\begin{aligned}
				A^TJ_iA&=J_i\\
				J_iA&=J_iA  \ \ for \ i=1,2,3,
			\end{aligned}
		\end{equation*}
		which means $J_i\widetilde{f}_*=\widetilde{f}_*J_i$ and $f^*(\rho_i)=\rho_i$ for $i=1,2,3$.
		Therefore, related to metric $g$ on $\mathbb{R}^{4n}$,
		\begin{equation*}
			\begin{aligned}
				g(\widetilde{f}_*(X),\widetilde{f}_*(Y))
				&=\rho_i(\widetilde{f}_*(X),J_i\widetilde{f}_*(Y))\\
				&=\rho_i(\widetilde{f}_*(X),\widetilde{f}_*(J_iY))\\
				&=(\widetilde{f}^*\rho_i)(X,J_iY)\\
				&=\rho_i(X,J_iY)\\
				&=g(X,Y).
			\end{aligned}
		\end{equation*} 
		$\widetilde{f}$ is isometric on $\mathbb{R}^{4n}$. So, the only possible is $\widetilde{f}$ is an affine transformation as well as $f=\widetilde{f}\circ Df(q)$.
		
	\end{proof}
	
	\begin{remark}
		If $f$ is a quasiconformal map on $\mathbb{R}^{4n}$ with a contact lift on $\mathbb{H}_\mathbb{H}^n$, then because of theorem \ref{thm-bi-Lips} it's bi-Lipschitz and $D_PF$ is full-rank. Therefore, this kind of $f$ would naturally satisfy the condition of theorem \ref{thm-R4n}.
	\end{remark}
	
	\begin{remark}
		For quaternionic Heisenberg group $\mathbb{H}_\mathbb{H}^n$, its Lie group structure has more delicate symmetrical structure. To be specific, continue to use the notations from the theorem \ref{thm-R4n},
		$\rho_1^{2n}=\rho_2^{2n}=\rho_3^{2n}=4^ndx_1\wedge\dots \wedge dx_n\wedge dy_1\wedge\dots \wedge dy_n\wedge dz_1\wedge\dots \wedge dz_n\wedge dw_1\wedge\dots \wedge dw_n$.
		From
		\begin{equation*}
			\widetilde{f}^*(\rho_1^{2n})=\widetilde{f}^*(\rho_2^{2n})=\widetilde{f}^*(\rho_3^{2n}),
		\end{equation*}
		we have 
		\begin{equation*}
			\widetilde{f}^*(\rho_1)^{2n}=\widetilde{f}^*(\rho_2)^{2n}=\widetilde{f}^*(\rho_3)^{2n}
		\end{equation*}
		so that  
		\begin{equation*}
			a_{11}^{2n}+a_{12}^{2n}+a_{13}^{2n}=a_{21}^{2n}+a_{22}^{2n}+a_{23}^{2n}=a_{31}^{2n}+a_{32}^{2n}+a_{33}^{2n}.
		\end{equation*}
	\end{remark}
	
	However, if the structure of $V$ in central extension is not rich enough, the contact lifts on the Carnot group are abundant.
	
	The two rigidity results above raise the complementary question of how many lifts exist when the fibre is small. The answer is that they are abundant, and the filiform group provides the cleanest illustration.
	
	\noindent\textbf{Example~\ref{ex-filiform}.}\label{ex-filiform} Consider the real filiform Carnot group $F^{3}$ of step $3$. The corresponding Lie algebra is given by
	\[
	\mathfrak{f}^{s} = \operatorname{span}(X_1, X_2) \oplus \operatorname{span}(X_{3}) \oplus \operatorname{span}(X_{4}),
	\]
	where $[X_1, X_2] = X_{3}$ and $[X_1, X_{3}] = X_{4}$ are the only non-trivial relations. Consider the corresponding dual basis $\omega_1, \omega_2,\omega_3, \omega_4$.
	
	\begin{equation*}
		\begin{aligned}
			&X_1=\frac{\partial}{\partial x_1}\\
			&X_2=\frac{\partial}{\partial x_2}+x_1\frac{\partial}{\partial x_3}+\frac{x_1^2}{2}\frac{\partial}{\partial x_4}\\
			&X_3=\frac{\partial}{\partial x_3}+x_1\frac{\partial}{\partial x_4}\\
			&X_4=\frac{\partial}{\partial x_4}
		\end{aligned}
	\end{equation*}
	and
	\begin{equation*}
		\begin{aligned}
			&\omega_1=dx_1\\
			&\omega_2=dx_2\\
			&\omega_3=-x_1dx_2+dx_3\\
			&\omega_4=\frac{x_1^2}{2}dx_2-x_1dx_3+dx_4
		\end{aligned}
	\end{equation*}
	
	(i) The easy way to find contact lifts is using central extension twice. The map $f:\mathbb{R}^2\rightarrow \mathbb{R}^2$ admits a contact lift $F$ on Heisenberg group $\mathbb{H}^1$ if and only if $f$ is a area-preserving map. If we want to lift $F$ onto $F^3$ as $G:F^3\rightarrow F^3$, its Pansu differential are supposed to be such form:
	\begin{equation*}
		D_PG=\begin{bmatrix}  
			a&  b \\
			0& \frac{1}{a^2} \\
			&  &\frac{1}{a}   \\
			&  &  &1
		\end{bmatrix}
	\end{equation*}
	where $a$ is any constant and for convenience we take $\varphi=id:span(X_4)\rightarrow span(X_4)$. Therefore, because $b$ can be any continuous function, we can find plenty of maps $f$ and $F$ satisfying the condition.     
	
	(ii)The second way is more complicate. By applying the same method in \cite{KoranyiReimann1985} as Kor\'anyi--Reimann theorem, we can use vector field to generate a one-parameter family of contact lift related to central extension $\mathbb{R}\rightarrow F^3\rightarrow F^2$.
	
	For vector field $V=p_1X_1+p_2X_2+p_3X_3+p_4X_4$, as a generator of contact lift, it need to satisfy $\mathcal{L}_V\omega_4=0$ for a group homomorphism as initial condition. As the motion along $V$, every $f_s$ would preserve $\omega_4$, which means it's a contact lift. 
	\begin{equation*}
		\begin{aligned}
			\mathcal{L}_V\omega_4&=d(\iota_V\omega_4)+\iota_V(d\omega_4)\\
			&=dp_4+p_3dx_1+x_1p_1dx_2-p_1dx_3\\
			&=(\frac{\partial}{\partial x_1}p_4+p_3)dx_1+(\frac{\partial}{\partial x_2}p_4+x_1p_1)dx_2+(\frac{\partial}{\partial x_3}p_4-p_1)dx_3+\frac{\partial}{\partial x_4}p_4dx_1
		\end{aligned}
	\end{equation*}
	Therefore, the vector field satisfies 
	\begin{equation}
		\left\{\begin{aligned}
			p_1&=\frac{\partial}{\partial x_3}p_4\\
			x_1p_1&=\frac{\partial}{\partial x_2}p_4\\
			p_3&=-\frac{\partial}{\partial x_1}p_4\\
			0&=\frac{\partial}{\partial x_4}p_4
		\end{aligned}
		\right.
	\end{equation}
	and this kind of vector fields are abundant.

	\section{The approximation of contact lift on Heisenberg group}\label{sec-approx}
	\subsection{Sobelev contact lift}
	The previous sections established criteria for the existence and rigidity of contact lifts for various regularity classes. A natural subsequent question is whether these (possibly rough) lifts can be approximated by smooth lifts, which are often easier to handle computationally and theoretically. This is particularly relevant for Sobolev and H\"older maps, where the lift $F$ may lack classical differentiability. In this section, we demonstrate that for the classical Heisenberg group $H^n$, any Sobolev or H\"older lift can be approximated by a family of smooth lifts $G_\varepsilon$, utilizing tools from symplectic geometry and mollification.
	
	Let us first record why this question is non-trivial. In the Sobolev category the lift is only defined up to null sets, and a mollification $f_{\varepsilon}$ of a map satisfying the closed-curve condition of Lemma~\ref{lem-Sobolev-body} will in general no longer satisfy it: the symplectic area enclosed by a curve is a quadratic expression in $f$, and quadratic expressions do not commute with convolution. The remedy is to correct the mollification by a diffeomorphism rather than to mollify more carefully, and Moser's trick provides exactly such a correction.
	
	We can construct a family of new contact lifts in two ways, the first is to Constructe through mapped composition.
	
	If $f_1,f_2$ admit contact lift $F_1,F_2$, then $F_2 \circ F_1$ is the contact lift of $f_2\circ f_1$, because we have the commutative diagram:
	\begin{displaymath}
		\xymatrix{
			V_1 \ar[r] \ar[d]_\Phi_{1} & G_1 \ar[d]^{ F_1} \ar[r]^{\pi_1} & H_1 \ar[d]^{ f_1}
			\\
			V_2 \ar[r] \ar[d]_\Phi_{2} & G_2 \ar[d]^{ F_2} \ar[r]^{\pi_2} & H_2 \ar[d]^{ f_2}
			\\
			V_3 \ar[r] & G_3  \ar[r]^{\pi_3} & H_3
		}
	\end{displaymath}
	and $F_2\circ F_1(gk)=F_2(F_1(g)\Phi_{1}(k))=F_2\circ F_1(g)\cdot \Phi_{2}\circ \Phi_{1}(k)$. It still holds under a limited number of composite operations.
	
	The second is to find a family of smooth contact maps which adimts smooth contact lifts as the approximation. However, because the the space of Lie algebra homomorphism is not a vector space, we are not able to use convolution to construct Pansu differential smooth approximation functions. However, we can study the lifting $F$ on Heisenberg group, due to $f$ as the continuous function on $\mathbb{R}^n$.
	
	\begin{lem}\label{lem-approx}
		Suppose $f$ and $f_\varepsilon$ in $W^{1,p}(H_1;H_2)$ admit contact lift $F$ and $F_\varepsilon$ in $W^{1,p}(G_1;G_2)$, if $f_\varepsilon\rightarrow f$ in $C^0$ then $F_\varepsilon\rightarrow F$ in $C^0$.
	\end{lem}
	\begin{proof}
		For Lipschitz curve $\widetilde{\gamma}$,
		\begin{equation}
			\begin{aligned}
				&d(F_\varepsilon(\widetilde{\gamma}(1)),F(\widetilde{\gamma}(1)))\\
				&=d(f(\gamma(1)) \cdot exp(\int_{\gamma}\left(f \circ \pi_{1}\right)^{*} \alpha_{2}),f_\varepsilon(\gamma(1)) \cdot exp(\int_{\gamma}\left(f_\varepsilon \circ \pi_{1}\right)^{*} \alpha_{2}))\\
				&\leq d(f(\gamma(1)),f_\varepsilon(\gamma(1)) +d( exp(\int_{\gamma}\left(f \circ \pi_{1}\right)^{*} \alpha_{2}), exp(\int_{\gamma}\left(f_\varepsilon \circ \pi_{1}\right)^{*} \alpha_{2}))\\
				&\rightarrow 0
			\end{aligned}
		\end{equation}
		
	\end{proof}
	
	For Heisenberg group $H_n$,
	\begin{displaymath}
		\xymatrix{
			R \ar[r] & H_n \ar[d]^{F} \ar[r]^{\pi} & \mathbb{R}^{2n} \ar[d]^{f}
			\\
			R \ar[r] & H_n  \ar[r]^{\pi} & \mathbb{R}^{2n}
		}
	\end{displaymath}
	Taking a Lipschitz curve $\gamma$ in $\mathbb{R}^{2n}$
	\begin{equation*}
		\begin{aligned}
			\gamma^z(1)-\gamma^z(0)&=\sum_{i=a}^{n}\int_{\gamma}x_i dy_i-y_i dx_i\\
			&=\sum_{i=a}^{n}\int_{D_i}dx_i\wedge dy_i,
		\end{aligned}
	\end{equation*}
	where $D_i$ is the area in the projection of the closed curve to $x_iOy_i$ plane which connects $\gamma$ and the line segment from 0 to $\gamma(1)$.
	If a sobolev map $f$ admits a lift $F$, then $f_*(Z)=\lambda Z$ for some constant $\lambda$, which is equivalent to $f^*(\omega)=\lambda \omega$ where $\omega$ is the standard symplectic form on $\mathbb{R}^{2n}$. The Pansu differential of $F$ is 
	\begin{equation*}
		D_PF=\begin{bmatrix}
			A &  \\
			& \lambda
		\end{bmatrix}
	\end{equation*}
	We can write the Lie algebra as follow:
	\begin{equation*}
		\begin{aligned}
			[X,Y]&=X^TJY\\
			\lambda[X,Y]&=f_*([X,Y])\\
			&=[f_*(X),f_*(Y)]\\
			&=[AX,AY]\\
			&=X^TA^TJAY
		\end{aligned}
	\end{equation*}
	where $J$ is the standard symplectic matrix 
	$J=\begin{bmatrix}
		0 & I_n \\
		-I_n & 0
	\end{bmatrix}$
	.
	So, $\lambda J=A^TJA$.
	\begin{equation*}
		\lambda^{2n}det(J)=det(A^T)det(J)det(A)
	\end{equation*}
	and $\lambda=(det(A))^{1/n}$.
	
	The next step is to construct a serious of smooth $g_\varepsilon$ which admit smooth contact lift $G_\varepsilon$ and $g_\varepsilon\rightarrow f$ as $\varepsilon\rightarrow 0$. For continuous $f$, we can use convolution to find smooth functions $f_\varepsilon$ and $f_\varepsilon\rightarrow f$ as $\varepsilon\rightarrow 0$, but for each $f_\varepsilon$ it can not preserve symplectic form. 
	
	using the Moser's trick to construct geometric flow, we define the error terms
	\begin{equation*}
		\eta_\varepsilon={f_\varepsilon}^*\omega-\omega.
	\end{equation*}
	We want to find diffeomorphism $\Psi_\varepsilon: \mathbb{R}^{2n}\rightarrow \mathbb{R}^{2n}$, such that $g_\varepsilon=f_\varepsilon \circ \Psi_\varepsilon$ and ${g_\varepsilon}^*\omega={\Psi_\varepsilon}^*({f_\varepsilon}^*\omega) =\omega$. Therefore, ${\Psi_\varepsilon}^*(\omega+\eta_\varepsilon)={\Psi_\varepsilon}^*({f_\varepsilon}^*\omega) =\omega$.
	
	Using vector field $X_{\varepsilon,t}$ define a flow $\phi_{\varepsilon,t}$, such that 
	\begin{equation*}
		\frac{d}{dt}(\phi_{\varepsilon,t})=X_{\varepsilon,t}\circ \phi_{\varepsilon,t} ,\  \phi_{\varepsilon,0}=id \ \text{and} \ \phi_{\varepsilon,1}=\Psi_\varepsilon
	\end{equation*}
	Considering $\omega_t=\omega+t\eta_\varepsilon$, $t\in [0,1]$, then diffeomorphism ${\phi_{\varepsilon,t}}$ such that 
	\begin{equation*}
		{\phi_{\varepsilon,t}}^*\omega_t=\omega
	\end{equation*}
	We need to find $X_{\varepsilon,t}$ satisfying $\mathcal{L}_{X_{\varepsilon,t}}\omega_t=-\eta_\varepsilon$, then 
	\begin{equation*}
		\begin{aligned}
			\frac{d}{dt}({\phi_{\varepsilon,t}}^*\omega_t)&={\phi_{\varepsilon,t}}^*(\mathcal{L}_{X_{\varepsilon,t}}\omega_t+\frac{d}{dt}\omega_t)\\
			&={\phi_{\varepsilon,t}}^*(-\eta_\varepsilon+\eta_\varepsilon)\\
			&=0
		\end{aligned}
	\end{equation*}
	and
	\begin{equation*}
		{\Psi_\varepsilon}^*(\omega+\eta_\varepsilon)={\phi_{\varepsilon,1}}^*\omega_1={\phi_{\varepsilon,0}}^*\omega_0=(id)^*\omega =\omega.
	\end{equation*}
	So, every $g_\varepsilon$ admit a smooth contact lift and $g_\varepsilon\rightarrow f$ as $\varepsilon\rightarrow 0$.
	The rest is to prove the equation  $\mathcal{L}_{X_{\varepsilon,t}}\omega_t=-\eta_\varepsilon$ is solvable.
	
	Because of $d\omega_t=0$, 
	\begin{equation*}
		\mathcal{L}_{X_{\varepsilon,t}}\omega_t=d(\iota_{X_{\varepsilon,t}}\omega_t)+\iota_{X_{\varepsilon,t}}(d\omega_t)=d(\iota_{X_{\varepsilon,t}}\omega_t),
	\end{equation*}
	it suffices to show $d(\iota_{X_{\varepsilon,t}}\omega_t)=-\eta_\varepsilon$ is solvable.
	For simply connected domain, closed from $\eta_\varepsilon$ is exact. Thus, there is a 1-from $\alpha_{\varepsilon,t}$ such that $\alpha_{\varepsilon,t}=\iota_{X_{\varepsilon,t}}\omega_t$. 
	For $\varepsilon$ small enough, $\eta_\varepsilon={f_\varepsilon}^*\omega-\omega$ is closed to $0$ in $L^p$ so that $\omega_t$ is closed to $\omega$ for any $t\in [0,1]$, thus $\omega_t$ is a a.e. non-degenerated 2-form and $X\rightarrow\iota_X\omega_t$ is a isomorphism from $TM$ to $T^*M$. Therefore, we can find $X_{\varepsilon,t}$ satisfying the condition in $L^p$.
	
	We now assemble the two ingredients. Mollification gives smooth maps $f_{\varepsilon}\to f$; Moser's trick corrects each $f_{\varepsilon}$ to a symplectomorphism $g_{\varepsilon}$; each $g_{\varepsilon}$ therefore admits a smooth contact lift $G_{\varepsilon}$; and Lemma~\ref{lem-approx} transfers the $C^{0}$-convergence of the base maps to the corresponding lifts. This proves Theorem~\ref{thm-approx-intro}.
	
	\begin{thm}\label{thm-approx-body}
		For any sobolev map $f$ on $\mathbb{R}^{2n}$ and its contact lift $F$ on $\mathbb{H}^n$, there are a serious of smooth $g_\varepsilon$ which admit smooth contact lift $G_\varepsilon$, $g_\varepsilon\rightarrow f$ and $G_\varepsilon\rightarrow F$ as $\varepsilon\rightarrow 0$.
	\end{thm}
	\begin{proof}
		The result follows by the proof above and Lemma \ref{lem-approx}.
	\end{proof}
	
	\subsection{H\"older lift}
	For H\"older lift, there is also a similar result about decomposion. At first, we need to clear the definition of H\"older lift. Without considering the differential structure, we just require the lifting $F$ of $f$ satisfying $f\circ\pi_{1}=\pi_{2}\circ F$ that is to say the following diagram commutes
	\begin{displaymath}
		\xymatrix{
			V_1 \ar[r]  & G_1 \ar[d]^{F} \ar[r]^{\pi_1} & H_1 \ar[d]^{f}
			\\
			V_2 \ar[r] & G_2  \ar[r]^{\pi_2} & H_2
		}
	\end{displaymath}
	
	Before diving into the lifting, we are supposed to show a lemma for H\"older continuous maps between Carnot group, which establishes the aniostropic structure of Carnot groups. The statement we are after is the H\"older substitute for the familiar fact that the Pansu differential of a Lipschitz map between Carnot groups is a graded Lie algebra homomorphism, and it will be obtained by mollification: the mollified pullback of a high-weight contact form, differentiated in a low-weight direction, has to decay because the two directions scale at incompatible rates.
	
	For convenience, we just take step three Carnot group as an example to showcase the result, as for higher step Carnot groups the conclusion holds as well. Throughout, we use the coordinates of the first kind, so that the group law and the translation formula \eqref{coord} are the objects we differentiate.
	
	Let  $G $ be a step three Carnot group. Let  $d_{1}, \ldots, d_{r}$  be a basis for  $\mathfrak{g}_{1}, e_{1}, \ldots, e_{s}$  a basis for  $\mathfrak{g}_{2} $, and  $f_{1}, \ldots, f_{t}$  a basis for  $\mathfrak{g}_{3}$ . We write
	\begin{equation*}
		\begin{array}{l}
			{\left[d_{i}, d_{j}\right]=\sum_{k=1}^{s} \alpha_{k}^{i j} e_{k}} \\
			{\left[d_{i}, e_{k}\right]=\sum_{m=1}^{t} \beta_{m}^{i k} f_{m}}
		\end{array}
	\end{equation*}
	with all other bracket relations trivial.
	As in the step two case, we can identify  $G$  with  $\mathbb{R}^{r+s+t}$  equipped with the following operation via coordinates of the first kind:
	\begin{equation*}
		\left(A_{i}, B_{k}, C_{m}\right) \star\left(a_{i}, b_{k}, c_{m}\right)=\left(\mathcal{A}_{i}, \mathcal{B}_{k}, \mathcal{C}_{m}\right)
	\end{equation*}
	where
	\begin{equation*}
		\begin{array}{l}
			\mathcal{A}_{i}=A_{i}+a_{i} \\
			\mathcal{B}_{k}=B_{k}+b_{k}+\frac{1}{2} \sum_{i<j} \alpha_{k}^{i j}\left(A_{i} a_{j}-a_{i} A_{j}\right) \\
			\mathcal{C}_{m}=C_{m}+c_{m}+\frac{1}{2} \sum_{i, j} \beta_{m}^{i j}\left(A_{i} b_{j}-B_{j} a_{i}\right)+\frac{1}{12} \sum_{l, k} \sum_{i<j}\left(A_{l}-a_{l}\right) \alpha_{k}^{i j}\left(A_{i} a_{j}-a_{i} A_{j}\right) \beta_{m}^{l k}
		\end{array}
	\end{equation*}
	$\operatorname{Observe}\left(A_{i}, B_{k}, C_{m}\right)^{-1}=\left(-A_{i},-B_{k},-C_{m}\right)$  and
	
	\begin{equation*}
		\left(A_{i}, B_{k}, C_{m}\right)^{-1} \star\left(a_{i}, b_{k}, c_{m}\right)=\left(\tilde{\mathcal{A}_{i}}, \tilde{\mathcal{B}_{k}}, \tilde{\mathcal{C}_{m}}\right),
	\end{equation*}
	where
	\begin{equation}\label{coord}
		\begin{array}{l}
			\tilde{\mathcal{A}}_{i}=a_{i}-A_{i} \\
			\tilde{\mathcal{B}_{k}}=b_{k}-B_{k}-\frac{1}{2} \sum_{i<j} \alpha_{k}^{i j}\left(A_{i} a_{j}-a_{i} A_{j}\right) \\
			\tilde{\mathcal{C}_{m}}=c_{m}-C_{m}-\frac{1}{2} \sum_{i, j} \beta_{m}^{i j}\left(A_{i} b_{j}-B_{j} a_{i}\right)+\frac{1}{12} \sum_{l, k} \sum_{i<j}\left(A_{l}+a_{l}\right) \alpha_{k}^{i j}\left(A_{i} a_{j}-a_{i} A_{j}\right) \beta_{m}^{l k}
		\end{array}
	\end{equation}
	Left-translating the canonical basis at the origin, we obtain the left-invariant vector fields
	\begin{equation*}
		\begin{aligned}
			X^{i}= & \frac{\partial}{\partial A_{i}}-\frac{1}{2}\sum_{k=1}^{s} \sum_{i<j} \alpha_{k}^{i j} A_{j} \frac{\partial}{\partial B_{k}} \\
			& +\sum_{m=1}^{t}\left[-\frac{1}{2} \sum_{j=1}^{s} B_{j} \beta_{m}^{i j}+\frac{1}{12} \sum_{l=1}^{r} \sum_{k=1}^{s} A_{l}\left(\sum_{i<j} \alpha_{k}^{i j} A_{j}\right) \beta_{m}^{l k}\right] \frac{\partial}{\partial C_{m}} \\
			Y^{k}= & \frac{\partial}{\partial B_{k}}+\sum_{m=1}^{t}\left(\frac{1}{2} \sum_{i=1}^{r} \beta_{m}^{i k} A_{i}\right) \frac{\partial}{\partial C_{m}} \\
			Z^{m}= & \frac{\partial}{\partial C_{m}}
		\end{aligned}
	\end{equation*}
	It is clear that  $\left\{X^{i}\right\}_{i} \cup\left\{Y^{k}\right\}_{k} \cup\left\{Z^{m}\right\}_{m}$  forms a basis for  $\operatorname{Lie}\left(\mathbb{R}^{n}, \star\right) $. Moreover, we have the expected step three stratification of $ \operatorname{Lie}\left(\mathbb{R}^{n}, \star\right) $ :
	\begin{equation*}
		\operatorname{Lie}\left(\mathbb{R}^{r+s+t}, \star\right)=\left\langle X^{i}\right\rangle_{1 \leq i \leq r} \oplus\left\langle Y^{k}\right\rangle_{1 \leq k \leq s} \oplus\left\langle Z^{m}\right\rangle_{1 \leq m \leq t}
	\end{equation*}
	In fact, one can show using the Jacobi identity that the linear map  $\varphi: \operatorname{Lie}\left(\mathbb{R}^{n}, \star\right) \rightarrow \mathfrak{g}$  induced by
	\begin{equation*}
		X^{i} \mapsto d_{i}, \quad Y^{k} \mapsto e_{k}, \quad Z^{m} \mapsto f_{m}
	\end{equation*}
	is a Lie algebra isomorphism.
	The contact forms are given by
	\begin{equation*}
		\begin{aligned}
			\omega_{1}^{k} & :=d B_{k}-\frac{1}{2}\sum_{i=1}^{r} \sum_{i<j} \alpha_{k}^{i j}\left( A_{i}d A_{j}- A_{j}d A_{i}\right)  \\
			\omega_{2}^{m} & :=d C_{m}+\frac{1}{2}\sum_{i=1}^{r} \sum_{j=1}^{s} B_{j} \beta_{m}^{i j}d A_{i}-\frac{1}{12} \sum_{l=1}^{r} \sum_{k=1}^{s}\sum_{i<j} A_{l} \alpha_{k}^{i j}\left( A_{i}d A_{j}- A_{j}d A_{i}\right) \beta_{m}^{l k} 
		\end{aligned}
	\end{equation*}
	We have
	\begin{equation*}
		H\left(\mathbb{R}^{r+s+t}, \star\right)=\bigcap_{k=1}^{s} \operatorname{ker} \omega_{1}^{k} \cap \bigcap_{m=1}^{t} \operatorname{ker} \omega_{2}^{m}
	\end{equation*}
	so that a tangent vector $ v $ lies in  $H_{p}\left(\mathbb{R}^{r+s+t}, \star\right)$  if and only if  $\left(\omega_{1}^{k}\right)_{p}(v)=\left(\omega_{2}^{m}\right)_{p}(v)=0$  for all $ k$  and  $m $.
	
	Every step $n$ Carnot group has the relation that
	if  $K \subset \mathbb{G}$  is compact, then there is a constant $ C=C(K) \geq 1 $ such that
	\begin{equation}\label{metric-compar}
		C^{-1}|p-q| \leq d_{c c}(p, q) \leq C|p-q|^{1 / n} \quad 
	\end{equation}
	for all $ p, q \in K$ .
	
	We can define the quasi-metric 
	\begin{equation}\label{metric-K}
		d_K(p,q)=\sum_{i=1}^{n}\sum_{j=1}^{dim(V_i)}|\pi_{i}^j(q^{-1}p)|^{\frac{1}{i}}
	\end{equation}
	where $\pi_{i}^j$ is the projection to the i-th layer and the j-th coordinate.
	\begin{equation*}
		\begin{aligned}
			&\pi_{1}^i(q^{-1}p)=a_{i}-A_{i}\\
			&\pi_{2}^k(q^{-1}p):=\varphi_{1}^k(p, q)=b_{k}-B_{k}-\frac{1}{2} \sum_{i<j} \alpha_{k}^{i j}\left(A_{i} a_{j}-a_{i} A_{j}\right)\\
			&\pi_{3}^m(q^{-1}p):=\varphi_{2}^m(p, q)=c_{m}-C_{m}-\frac{1}{2} \sum_{i, j} \beta_{m}^{i j}\left(A_{i} b_{j}-B_{j} a_{i}\right)+\frac{1}{12} \sum_{l, k} \sum_{i<j}\left(A_{l}+a_{l}\right) \alpha_{k}^{i j}\left(A_{i} a_{j}-a_{i} A_{j}\right) \beta_{m}^{l k}
		\end{aligned}
	\end{equation*}
	
	It is worth to mention that $d_K$ is comparable with C-C metric on Carnot group.
	
	It is known that for Pansu differentiable map $f$  the Pansu differential is a Lie algebra homorophism. Thus, if the vector field $X$ and contact form $\omega$ are not from the same layer, then $X(f_P^*\omega)=0$. For instance, on Heisenberg group $\mathbb{H}^1$
	\begin{equation*}
		\begin{aligned}
			D_Pf
			&=\begin{bmatrix}
				X_1(f_P^*dx_1) & X_1(f_P^*dx_2) & 0\\
				X_2(f_P^*dx_1) & X_2(f_P^*dx_2) & 0\\
				0 & 0 & Z(f_P^*\omega_{1})\\
			\end{bmatrix}
		\end{aligned}
	\end{equation*}
	Although the Pansu differential of H\"older continuous maps do not exist, we can we can study the existence of differential in the sense of Riemannian metric with the help of mollification.
	
	\begin{thm}\label{thm-vanish}
		Suppose that  $f \in C^{0, \beta}\left(\Omega ; G_2\right)$ , where $ G_2$ is a step $n_2$ Carnot group , $\Omega \subset G_1$  is open , $X$ is a left-invariant  vector field belonging to the horizontal layer and $ 0<\beta \leq 1$ . If  $B\left(x_{o}, 2 r\right) \subset \Omega $, then  for all  $0<\varepsilon<r$,
		\begin{equation}\label{vanish}
			\begin{aligned}
				&X_1(f_{\varepsilon}^{*} \omega_{1}^k) \lesssim  \varepsilon^{2 \beta-1} \\
				&X_1(f_{\varepsilon}^{*} \omega_{2}^m) \lesssim  \varepsilon^{2 \beta-1}\\
				&\dots \\
				&X_1(f_{\varepsilon}^{*} \omega_{n_2}^m) \lesssim  \varepsilon^{2 \beta-1}
			\end{aligned}
		\end{equation}
		where  $f_{\varepsilon}=f * \eta_{\varepsilon}=\int_{\mathbb{G}_1 } f(y) \eta_{\varepsilon}\left(y^{-1} x\right) d y=\int_{\mathbb{G}_1 } f(xy^{-1}) \eta_{\varepsilon}\left(y\right) d y$  and $\eta_{\varepsilon}$ is a smooth function on Carnot group $G_1$ with a compact supported set $B(e,\varepsilon)$ and
		\begin{equation*}
			\int_{G_1}\eta_{\varepsilon}=1.
		\end{equation*}
		\begin{proof}
			Recall that  $\varphi(p, q)$  was defined in \eqref{metric-K}. Let  $p \in B\left(x_{o}, r\right)$  and let  $\varepsilon \in(0, r)$ . In what follows we will identify  $\left(f_{\varepsilon}^{*} \alpha\right)(p)$  with the vector (of equal length):
			\begin{equation*}
				\begin{aligned}
					&X(f_{\varepsilon}^{*} \omega_{1}^k)(p)\\
					=&X( f_{\varepsilon}^{B_{k}})(p)+\frac{1}{2}\alpha_{k}^{ij} \sum_{i<j}\left(f_{\varepsilon}^{A_{j}}(p) X(f_{\varepsilon}^{A_{i}})(p)-f_{\varepsilon}^{A_{i}}(p) X( f_{\varepsilon}^{A_{j}})(p)\right)\\
					=& 
					\varepsilon^{-\nu_1-1} \int_{B_{\varepsilon}}\left(f^{B_{k}}(pz^{-1})-f_{\varepsilon}^{B_{k}}(p)+\frac{1}{2}\alpha_{k}^{ij} \sum_{i<j}\left(f_{\varepsilon}^{A{j}}(p) f^{A_{i}}(pz^{-1})-f_{\varepsilon}^{A_{i}}(p) f^{A_{j}}(pz^{-1})\right)\right)\\
					&\cdot X(\eta\left(\delta_{\varepsilon^{-1}}(z)\right)) d z \\
					=&\varepsilon^{-\nu_1-1} \int_{B_{\varepsilon}} \varphi_{k}\left(f(pz^{-1}), f_{\varepsilon}(p)\right) X(\eta\left(\delta_{\varepsilon^{-1}}(z)\right)) d z \\
					=&\varepsilon^{-\nu_1-1} \int_{B_{\varepsilon}} \int_{B_{\varepsilon}} \varphi_{k}(f(pz^{-1}), f(pw^{-1})) \eta_{\varepsilon}(w) X(\eta\left(\delta_{\varepsilon^{-1}}(z)\right)) d w d z
				\end{aligned}
			\end{equation*}
			In the second equality we used the fact that $ \int_{B_{\varepsilon}} f_{\varepsilon}^{t_{k}}(p) X( \eta\left(\delta_{\varepsilon^{-1}}(z)\right)) d z=0 $. Easy verification of the last equality is left to the reader. Note that \eqref{metric-K} yields
			\begin{equation*}
				|\varphi_{k}(f(pz^{-1}), f(wz^{-1}))| \leq d_{K}(f(pz^{-1}), f(wz^{-1}))^{2} \leq[f]_{\gamma, 2 \varepsilon, B\left(x_{o}, 2 r\right)}^{2}(2 \varepsilon)^{2 \beta}
			\end{equation*}
			Therefore,
			\begin{equation}\label{v}
				\left|X(f_{\varepsilon}^{*} \omega_{1}^k)(p)\right| 
				\leq[f]_{\beta, 2 \varepsilon, B\left(x_{o}, 2 r\right)}^{2}(2 \varepsilon)^{2 \beta} \varepsilon^{-1} \varepsilon^{-\nu_1} \int_{B_{\varepsilon}}\left|X( \eta\left(\delta_{\varepsilon^{-1}}(z)\right))\right| d z
				\lesssim \varepsilon^{2 \beta-1}
			\end{equation}
			
			The following method is enlightened by \cite{Jung2017}.
			Similarly, to prove $\left|X(f_{\varepsilon}^{*} \omega_{2}^m)(p)\right|\lesssim  \varepsilon^{2 \beta-1}$ we would construct the equation have the form like 
			\begin{equation*}
				\begin{array}{c}
					\varphi_{2}^m (\left(a_{i}, b_{k}, c_{m}\right),\left(A_{i}, B_{k}, C_{m}\right))=\tilde{\mathcal{C}_{m}}\\
					=c_{m}-C_{m}-\frac{1}{2} \sum_{i, j} \beta_{m}^{i j}\left(A_{i} b_{j}-B_{j} a_{i}\right)+\frac{1}{12} \sum_{l, k} \sum_{i<j}\left(A_{l}+a_{l}\right) \alpha_{k}^{i j}\left(A_{i} a_{j}-a_{i} A_{j}\right) \beta_{m}^{l k}.
				\end{array}
			\end{equation*}
			It is easy to find the following result.
			\begin{equation*}
				\begin{array}{c}
					X(f_{\varepsilon}^{*} \omega_{2}^m)+X\left(f_{\varepsilon}^{*}\left(-\frac{1}{2} \sum_{i, j} \beta_{m}^{i j}A_{i} dB_{j}\right)\right)+X\left(f_{\varepsilon}^{*}\left(\frac{1}{12} \sum_{l, k} \sum_{i<j}A_{l} \alpha_{k}^{i j}\left(A_{i} dA_{j}-A_{j} dA_{i}\right) \beta_{m}^{l k} \right)\right)\\
					=\varepsilon^{-m-1} \int_{B_{\varepsilon}}\int_{B_{\varepsilon}}\left(f^{C_{m}}(pz^{-1})-f^{C_{m}}(pv^{-1})\right)\eta_{\varepsilon}(v)X( \eta(\delta_{\varepsilon^{-1}}(z)))d zdv\\
					-\frac{1}{2} \sum_{i, j} \beta_{m}^{i j}\varepsilon^{-m-1} \int_{B_{\varepsilon}}\int_{B_{\varepsilon}}\left(f_{\varepsilon}^{A_{i}}(pv^{-1}) f^{B_{j}}(pz^{-1})-f_{\varepsilon}^{B_{j}}(pv^{-1}) f^{A_{i}}(pz^{-1})\right) \\
					\eta_{\varepsilon}(v)X( \eta(\delta_{\varepsilon^{-1}}(z)))d zdv\\
					+\frac{1}{12} \sum_{l, k} \sum_{i<j} \beta_{m}^{l k}\alpha_{k}^{i j}\int_{B_{\varepsilon}}\int_{B_{\varepsilon}}\int_{B_{\varepsilon}}(f^{A_{l}}(pv^{-1})\left(f^{A_{i}}(pv^{-1})f^{A_{j}}(pz^{-1})- f^{A_{j}}(pv^{-1})f^{A_{i}}(pz^{-1})\right)\\ \eta_{\varepsilon}(w)\eta_{\varepsilon}(v)X( \eta(\delta_{\varepsilon^{-1}}(z)))d zdvdw\\
					+\frac{1}{12} \sum_{l, k} \sum_{i<j} \beta_{m}^{l k}\alpha_{k}^{i j}\int_{B_{\varepsilon}}\int_{B_{\varepsilon}}\int_{B_{\varepsilon}}(f^{A_{l}}(pz^{-1})\left(f^{A_{i}}(pv^{-1})f^{A_{j}}(pz^{-1})- f^{A_{j}}(pv^{-1})f^{A_{i}}(pz^{-1})\right)\\ \eta_{\varepsilon}(v)\eta_{\varepsilon}(w)X( \eta(\delta_{\varepsilon^{-1}}(z)))d wdzdv\\
					:=\mathbb{A}
				\end{array}
			\end{equation*}
			Since 
			\begin{equation*}
				\begin{aligned}
					\left|f_\varepsilon^{A_i} (p)f_\varepsilon^{A_j} (p)-(f^{A_i}(p)f^{A_j}(p))_\varepsilon\right|&\le \int_{B_{\varepsilon}}\int_{B_{\varepsilon}}\left|f^{A_i}(pz^{-1})\right|\left|f^{A_i}(pv^{-1})-f^{A_j}(pz^{-1})\right|\eta_\varepsilon(z)\eta_\varepsilon(v)dzdv\\
					&\lesssim
					\int_{B_{\varepsilon}}\int_{B_{\varepsilon}}\left|f^{A_i}(pz^{-1})\right|d_{cc}(f^{A_i}(pv^{-1}),f^{A_j}(pz^{-1}))\eta_\varepsilon(z)\eta_\varepsilon(v)dzdv\\
					&\le\int_{B_{\varepsilon}}\int_{B_{\varepsilon}}\varepsilon^\beta\left|f^{A_i}(pz^{-1})\right|\eta_\varepsilon(z)\eta_\varepsilon(v)dzdv\\
					&\lesssim\varepsilon^\beta
				\end{aligned}
			\end{equation*}
			we have
			\begin{equation*}
				\begin{aligned}
					& | \int_{B_{\varepsilon}}\int_{B_{\varepsilon}}\int_{B_{\varepsilon}}f^{A_{l}}(pw^{-1})\left(f^{A_{i}}(pv^{-1})f^{A_{j}}(pz^{-1})- f^{A_{j}}(pv^{-1})f^{A_{i}}(pz^{-1})\right)\eta_{\varepsilon}(v)\eta_{\varepsilon}(w)X( \eta(\delta_{\varepsilon^{-1}}(z)))d zdwdv\\
					&-\int_{B_{\varepsilon}}\int_{B_{\varepsilon}}f^{A_{l}}(pv^{-1})\left(f^{A_{i}}(pv^{-1})f^{A_{j}}(pz^{-1})- f^{A_{j}}(pv^{-1})f^{A_{i}}(pz^{-1})\right)\eta_{\varepsilon}(v)X( \eta(\delta_{\varepsilon^{-1}}(z)))d zdv | \\
					&
					\lesssim \int_{B_{\varepsilon}}\int_{B_{\varepsilon}}\varepsilon^\beta\left|\left(f^{A_{i}}(pv^{-1})f^{A_{j}}(pz^{-1})- f^{A_{j}}(pv^{-1})f^{A_{i}}(pz^{-1})\right)\right|\eta_{\varepsilon}(v)X( \eta(\delta_{\varepsilon^{-1}}(z)))d zdv
				\end{aligned}
			\end{equation*}
			Thus, we have
			\begin{equation*}
				\begin{aligned}
					\mathbb{A}&=\varepsilon^{-m-1} \int_{B_{\varepsilon}}\int_{B_{\varepsilon}}\varphi_{2}^m (f(pz^{-1}),f(pv^{-1}))\eta_{\varepsilon}(v)X( \eta(\delta_{\varepsilon^{-1}}(z)))dzdv\\
					&+\frac{1}{12} \sum_{l, k} \sum_{i<j} \beta_{m}^{l k}\alpha_{k}^{i j}\int_{B_{\varepsilon}}\int_{B_{\varepsilon}}\varepsilon^\beta\left(f^{A_{i}}(pv^{-1})f^{A_{j}}(pz^{-1})- f^{A_{j}}(pv^{-1})f^{A_{i}}(pz^{-1})\right)\eta_{\varepsilon}(v)X( \eta(\delta_{\varepsilon^{-1}}(z)))d zdv\\
					&\lesssim \varepsilon^{-m-1} \int_{B_{\varepsilon}}\int_{B_{\varepsilon}}\varphi_{2}^m (f(pz^{-1}),f(pv^{-1}))\eta_{\varepsilon}(v)X( \eta(\delta_{\varepsilon^{-1}}(z)))dzdv\\
					&+\frac{1}{6}\varepsilon^\beta \sum_{l, k}  \beta_{m}^{l k}\int_{B_{\varepsilon}}\int_{B_{\varepsilon}}(\left|\varphi_1^k(f(pz^{-1}),f(pv^{-1}))\right|+\left|f^{B_k}(pz^{-1})-f^{B_k}(pv^{-1})\right| )\eta_{\varepsilon}(v)X( \eta(\delta_{\varepsilon^{-1}}(z)))d zdv\\
					&\lesssim \varepsilon^{-m-1} \int_{B_{\varepsilon}}\int_{B_{\varepsilon}}\varphi_{2}^m (f(pz^{-1}),f(pv^{-1}))\eta_{\varepsilon}(v)X( \eta(\delta_{\varepsilon^{-1}}(z)))dzdv\\
					&+\frac{1}{6}\varepsilon^\beta \sum_{l, k}  \beta_{m}^{l k}\int_{B_{\varepsilon}}\int_{B_{\varepsilon}}(\left|\varphi_1^k(f(pz^{-1}),f(pv^{-1}))\right|+d_{cc}(f(pz^{-1}),f(pv^{-1})))\eta_{\varepsilon}(v)X( \eta(\delta_{\varepsilon^{-1}}(z)))d zdv\\
					&\lesssim\varepsilon^{-m-1} \int_{B_{\varepsilon}}\int_{B_{\varepsilon}}\varphi_{2}^m (f(pz^{-1}),f(pv^{-1}))\eta_{\varepsilon}(v)X( \eta(\delta_{\varepsilon^{-1}}(z)))dzdv\\ &+\varepsilon^{3\beta-1}+\varepsilon^{2\beta-1}
				\end{aligned}
			\end{equation*}
			
			Since we have
			\begin{equation*}
				\left|\varphi_{2}^m (f(pz^{-1}),f(pv^{-1}))\right|\lesssim d_{cc}(f(pz^{-1}),f(pv^{-1}))^3\lesssim [f]_{\gamma, 2 \varepsilon, B^{m}\left(x_{o}, 2 r\right)}^{3}(2 \varepsilon)^{3 \beta},
			\end{equation*}
			so
			\begin{equation*}
				\begin{array}{c}
					\left|X(f_{\varepsilon}^{*} \omega_{2}^m)+X\left(f_{\varepsilon}^{*}\left(-\frac{1}{2} \sum_{i, j} \beta_{m}^{i j}A_{i} dB_{j}\right)\right)+X\left(f_{\varepsilon}^{*}\left(\frac{1}{12} \sum_{l, k} \sum_{i<j}A_{l} \alpha_{k}^{i j}\left(A_{i} dA_{j}-A_{j} dA_{i}\right) \beta_{m}^{l k} \right)\right)\right|\\
					\le\varepsilon^{-m-1} \int_{B_{\varepsilon}}\int_{B_{\varepsilon}}\int_{B_{\varepsilon}}\left|\varphi_{2}^m (f(pz^{-1}),f(pv^{-1}))\right|\eta_{\varepsilon}(v)\eta_{\varepsilon}(z)\left|X( \eta(\delta_{\varepsilon^{-1}}(z)))\right|d zdzdv+\varepsilon^{2\beta-1}\\
					\lesssim \varepsilon^{2 \beta-1}
				\end{array}
			\end{equation*}
			Taking use of 
			$\left\|X(f_{\varepsilon}^{*} \omega_{1}^k)\right\|_{L^{\infty}\left(B^{m}\left(x_{o}, r\right)\right)} \lesssim  \varepsilon^{2 \beta-1}$ 
			we can get
			\begin{equation*}
				\begin{array}{c}
					\left|X(f_{\varepsilon}^{*} \omega_{2}^m)\right|\lesssim \varepsilon^{2 \beta-1}+X\left(f_{\varepsilon}^{*}\left(-\frac{1}{2} \sum_{i, j} \beta_{m}^{i j}A_{i} dB_{j}\right)\right)+X\left(f_{\varepsilon}^{*}\left(\frac{1}{12} \sum_{l, k} \sum_{i<j}A_{l} \alpha_{k}^{i j}\left(A_{i} dA_{j}-A_{j} dA_{i}\right) \beta_{m}^{l k} \right)\right)\\
					\lesssim \varepsilon^{2 \beta-1}+\left(\frac{5}{12} \sum_{l, k} \sum_{i<j} \beta_{m}^{l k}f_{\varepsilon}^{A(i)}(p)X f_{\varepsilon}^{B(j)}(p) \right)+\left(\frac{1}{12} \sum_{l, k} \sum_{i<j} \beta_{m}^{l k}f_{\varepsilon}^{A(i)}(p)\varepsilon^{2 \beta-1} \right)
				\end{array}
			\end{equation*}
			For $X(f_\varepsilon^{B(j)})(p)$,
			\begin{equation}
				X(f_\varepsilon^{B(j)})\leq \left|X(f_\varepsilon^{*}\omega_1^k)(p)\right|+\left|X(f_\varepsilon^{*}(\frac{1}{2}\sum_{i<j}\alpha_{k}^{ij}(A_idA_j-A_jdA_i)))\right|
			\end{equation}
			From \eqref{v} , we know that $\left|X(f_\varepsilon^{B(j)})(p)\right|\lesssim \varepsilon^{2 \beta-1} $
			\begin{equation*}
				\left|f_{\varepsilon}^{*} \omega_{2}^m\right|\lesssim \varepsilon^{2 \beta-1}+\varepsilon^{2 \beta-1}
			\end{equation*}
			
			For higher step Carnot groups, we can similarly prove that $X_1(f_{\varepsilon}^{*} \omega_{3}^k) \lesssim  \varepsilon^{2 \beta-1},
			\dots ,
			X_1(f_{\varepsilon}^{*} \omega_{n_2}^m) \lesssim  \varepsilon^{2 \beta-1}$.
			
		\end{proof}
	\end{thm}
	\begin{remark}
		The $\beta>\frac{1}{2}$ in theorem \ref{thm-vanish} is sharp. For $\Omega \hookrightarrow \mathbb{H}^n$ a domain on $\mathbb{R}^m$ H\"older embedding into Heisenberg group, when $\beta<\frac{1}{2}$, we can construct a $\beta$-H\"older deformation with H\"older homotopic map to make the image $f(\Omega)$ to shrink to origin. We refer the reader to section $10$ in \cite{HMS2025}.
	\end{remark}
	
	From the theorem \ref{thm-vanish}, we know that for $f \in C^{0, \beta}\left(\Omega ; G_2\right)$ and $ 1/2<\beta \leq 1$ the pull-back of the left-invariant forms will decay to 0 at a rate proportional to $\varepsilon^{2 \beta-1}$, when $i\geq1$ for $X(f_{\varepsilon}^{*} \omega_{i}^k)  $ which is one of the projections of $f_{\varepsilon}^{*} \omega_{i}^k$ to the horizontal layer. When considering the pansu differential for Lipschitz mappings, the pansu differential is a Lie algebra homomorphism so is a block diagonal matrix. Thus, we can get the similar result for H\"older mappings.

	\begin{thm}\label{thm-decomp}
		Let $V_1 \to G_1 \to H_1$ and $V_2 \to G_2 \to H_2$ be central extensions of Carnot groups such that $\mathrm{rank}(G_2) = \mathrm{rank}(H_2)$. Let $U_1 \subset H_1$ be a domain and $F: \pi_1^{-1}(U_1) \to G_2$ a H\"older lift of a H\"older map $f: U_1 \to H_2$. Then there exists a Lie group homomorphism $\Phi: V_1 \to V_2$ such that $F(gk) = F(g)\Phi(k)$ for all $g \in \pi_1^{-1}(U_1)$ and $k \in V_1$.
	\end{thm}
	\begin{proof}
		By the lift assumption, we have
		\[
		\pi_2 \circ F(gk) = f \circ \pi_1(gk) = f \circ \pi_1(g) = \pi_2 \circ F(g).
		\]
		Thus, for each $g \in \pi_1^{-1}(U_1)$ and $k \in V_1$, there exists some $\tilde{\Phi}(g,k) \in V_2$ such that $F(gk) = F(g)\tilde{\Phi}(g,k)$. We claim that $\tilde{\Phi}(g,k)$ is constant in $g \in \pi_1^{-1}(U_1)$. That is, there is a well defined map $\Phi: V_1 \to V_2$ such that $\tilde{\Phi}(g,k) = \Phi(k)$ for $g \in \pi_1^{-1}(U_1)$ and $k \in V_1$. By assumption, $U_1$ is connected, so it suffices to prove that $g \mapsto \tilde{\Phi}(g,k)$ is locally constant for fixed $k \in V_1$.
		
		Let $B'$ be a ball that is compactly contained in $\pi_1^{-1}(U_1)$. Let $B$ be a ball with the same center and half the radius.
		
		Let $g, h \in B$. Since the extensions are central, elements of $V_1$ commute with elements of $G_1$, and similarly for $V_2$. By left-invariance of the distance, we compute that
		\[
		\begin{aligned}
			d(\tilde{\Phi}(g,k), \tilde{\Phi}(h,k)) &= d(F(g)\tilde{\Phi}(g,k), F(g)\tilde{\Phi}(h,k)) \\
			&\leq d(F(g)\tilde{\Phi}(g,k), F(h)\tilde{\Phi}(h,k)) \\
			&\quad + d(F(h)\tilde{\Phi}(h,k), F(g)\tilde{\Phi}(h,k)) \\
			&= d(F(gk), F(hk)) + d(F(h), F(g)) \\
			&\leq Ld(gk,hk)^\beta+Ld(h,g)^\beta\\
			&=2Ld(h,g)^\beta
		\end{aligned}
		\]
		That means $\tilde{\Phi}(g,k)\in C^{0,\beta}$.
		By the equal rank assumption, $V_2$ does not contain any horizontal directions.
		
		\textbf{Claim:}\ When $\beta>\frac{1}{2}$, the only possible H\"older maps $B \to V_2$ are constants.
		
		By contradiction, assume a H\"older maps $h:B \to V_2$, for $\beta>\frac{1}{2}$, then there are two points $p,q$ such that $h(p)\neq h(q)$. We take a family of smooth approximation $h_\varepsilon$ of $h$ as in theorem \ref{thm-vanish}.
		
		take a smooth horizontal curve $\alpha:[0,1]\rightarrow G_1$ connecting $p$ and $q$, $\alpha(0)=p$ and $\alpha(1)=q$. The tangent vector of $\alpha$ is
		\begin{equation*}
			\alpha'(t)=\sum_{i=1}^{dim(\mathfrak{g}_1^1)}a_i(t)X_i,
		\end{equation*}
		where the $X_i$ are horizontal basis of $\mathfrak{g}_1$. 
		\begin{equation*}
			\begin{aligned}
				h_{\varepsilon*}(\alpha'(t))&=\sum_{i=1}^{dim(\mathfrak{g}_1^1)}a_i(t)h_{\varepsilon*}(X_i)\\
				&=\sum_{i=1}^{dim(\mathfrak{g}_1^1)}a_i(t)\sum_{j=1}^{dim(V_2)}X_ih_{\varepsilon}^*(\omega_j')X_j',
			\end{aligned}
		\end{equation*}
		where the $X_j'$ and $\omega_j'$ are the vectors and dual vectors of $V_2$. 
		Therefore,
		\begin{equation*}
			\left|h_{\varepsilon*}(\alpha'(t))\right|=\sum_{i=1}^{dim(\mathfrak{g}_1^1)}\sum_{j=1}^{dim(V_2)}|a_i(t)|\left|X_ih_{\varepsilon}^*(\omega_j')\right|\left|X_j'\right|\lesssim \varepsilon^{2\beta-1}
		\end{equation*}
		This is because theorem \ref{thm-vanish}. 
		
		Now, under Euclidean metric $d$,
		\begin{equation*}
			\begin{aligned}
				d(h(p),h(q))&\leq d(h(p),h_\varepsilon(p))+d(h_\varepsilon(p),h_\varepsilon(q))+d(h_\varepsilon(q),h(q))\\
				&\leq d(h_\varepsilon(p),h_\varepsilon(q))+2C(\varepsilon)\\
				&=2C(\varepsilon)+\int_{0}^{1}\left|h_{\varepsilon*}(\alpha'(t))\right|dt\\
				&\lesssim2C(\varepsilon)+\varepsilon^{2\beta-1}\rightarrow 0 ,
			\end{aligned}
		\end{equation*}
		where the second line is because $h_\varepsilon\rightrightarrows h$, then there is a constant $C(\varepsilon)$ such that $d(h(p),h_\varepsilon(p))\leq C(\varepsilon)$ for all $p\in G_1$. Thus $h(p)=h(q)$, a contradiction.
		
		This claim can also be written in the following sense, as a general sense of Pansu differential. If $\beta>1/2$, by theorem \ref{thm-vanish}, then for any horizontal vector field $X$ 
		\begin{equation*}
			\pi_{{g_2}^[i]}\widetilde{\Phi}_{*}(X)=0     , \ \ \  \text{for} \ \ \  i \geq 2
		\end{equation*}
		where $\pi_{{g_2}^[i]}$ are the projection to the i-th layer. Although the Pansu differential of $F$ does not exist, the limit still has a structure similar to Lie algebra homorophism. 
		
		Consequently, $g \mapsto \tilde{\Phi}(g,k)$ is locally constant for $k \in V_1$.
		
		Next, we show that $\Phi: V_1 \to V_2$ is a group homomorphism. Let $k_1, k_2 \in V_1$ and let $g \in \pi_1^{-1}(U_1)$. Applying the definition of $\tilde{\Phi}$ with the pair $(g, k_1 k_2)$ shows
		\[
		F(g k_1 k_2) = F(g) \Phi(k_1 k_2),
		\]
		whereas applying the definition with the pairs $(g k_1, k_2)$ and $(g, k_1)$ yields
		\[
		F(g k_1 k_2) = F(g k_1) \Phi(k_2) = F(g) \Phi(k_1) \Phi(k_2).
		\]
		Canceling out $F(g)$ implies that $\Phi$ is a group homomorphism.
		
		Finally, we observe that $\Phi(k) = F(g)^{-1}F(gk)$ for any $g \in G_{1}$ and $k \in V_{1}$. 
	\end{proof}
	
	Especially, we can also use a family of smooth contact lift to approximate a  H\"older  lift when the case is on Heisenberger group $H_1$. The mechanism is identical to the Sobolev case of Theorem~\ref{thm-approx-body}, except that the area-preservation of $f$ is now obtained from the fibre decomposition of Theorem~\ref{thm-decomp} rather than from the a.e. existence of the Pansu differential.
	
	\begin{thm}\label{thm-holder-approx-body}
		Suppose that $f\in C^{0,\alpha}(\mathbb{R}^{2};\mathbb{R}^{2})$ is a H\"older map, for $\alpha>1/2$, and it admits a lift $F\in C^{0,\alpha}(H_1; H_1)$, then there are a family of smooth contact maps $G_\varepsilon $  on $H_1$ and $G_\varepsilon \rightarrow F$ as $\varepsilon \rightarrow 0$ .
	\end{thm}
	\begin{proof}
		For central extension
		\begin{displaymath}
			\xymatrix{
				R \ar[r] & H_1 \ar[d]^{F} \ar[r]^{\pi} & \mathbb{R}^2 \ar[d]^{f}
				\\
				R \ar[r] & H_1  \ar[r]^{\pi} & \mathbb{R}^2
			}
		\end{displaymath}
		by the assumption $\alpha>1/2$, the $\int_{f\circ \gamma}xdy-ydx$ exists for any Lipschitz curve $\gamma$ in $\mathbb{R}^2$.
		\begin{equation*}
			\begin{aligned}
				\int_{f\circ \gamma}xdy-ydx&=\int_{\gamma}f^*(xdy-ydx)\\
				&=\int_{\gamma}f_xdf_y-f_ydf_x
			\end{aligned}
		\end{equation*}
		
		By theorem \ref{thm-decomp}, there is a Lie group homomorphism $\Phi: V_1 \to V_2$ such that $F(gk) = F(g)\Phi(k)$, then $\Phi(\gamma_z(1))=\lambda \gamma_z(1)$ for some constant $\lambda$, where $\gamma(0)=0$ and 
		\begin{equation*}
			\begin{aligned}
				\gamma_z(1)&=\int_{\gamma}xdy-ydx\\
				&\overset{Stokes}{=}\int_{D}dS.
			\end{aligned}
		\end{equation*}
		Because $\gamma$ is arbitrary, $f$ maps any simply connected domain to a simply connected domain with an area of $\lambda$ times. For convenience, We can assume $\lambda=1$, so $f$ is area preserving map. Defining measure:
		\begin{equation*}
			\mu_f(D):=m(f(D)),
		\end{equation*}
		then $\mu_f$ is absolutely continuous related to Lebesgue measure on $\mathbb{R}^2$. Thus, we can define determinant $det(J_f)=1$ almost everywhere.
		
		For $f$, we can use convolution to find a family of smooth functions $f_\varepsilon$ and $f_\varepsilon\rightarrow f$ as $\varepsilon \rightarrow 0$. By using the same method as in the proof of Theorem~\ref{thm-approx-body}, we can use diffeomorphism $\Psi_\varepsilon: \mathbb{R}^{2}\rightarrow \mathbb{R}^{2}$ to make every $g_\varepsilon=f_\varepsilon\circ \Psi_\varepsilon$ is area preserving.
		Therefore, every smooth $g_\varepsilon$ admits a smooth contact lift $G_\varepsilon$ and $G_\varepsilon \rightarrow F$ as $\varepsilon \rightarrow 0$.
	\end{proof}
	
	Based on the result above, every H\"older lift on Heisenberg group $H_1$, we can use a family of smooth contact lift to approximate. Conversely, we can use a family of smooth contact lift to construct a H\"older lift on Heisenberg group $H_1$; this converse is the content of the next section, where we give necessary and sufficient conditions for a H\"older map to admit a H\"older lift at all.
	
	\section{The existence of H\"older lift}\label{sec-holder-lift}
	In this section, we connect the existence of H\"older lifts to a condition on closed horizontal curves.
	
	We can draw a similar conclusion as Lipschitz version just for Sobolev contact lift. The only genuine difficulty is that the lift is defined by an integral
	\begin{equation*}
		F(\widetilde{\gamma}(1))=f(\gamma(1))\cdot\exp\Bigl(\int_{\gamma}(f\circ\pi_{1})^{*}\alpha_{2}\Bigr),
	\end{equation*}
	and for a H\"older map this integral is not a classical Lebesgue integral: both the integrand and the curve may fail to be of bounded variation. Note that the horizontal estimate \eqref{vanish} is precisely what makes the integrand sufficiently regular for a Young-type theory to apply. The remedy is a Young-type theory of integration, whose basic estimate we record next.
	
	At first, we are supposed to generalize the integral for H\"older continuous function. The proposition 3.13, lemma 3.15 and corollary 3.16 in \cite{HMS2025} tell us we can generalize the the integral for H\"older condition. And also, it is easy to check the result can be used to any $\alpha,\beta>0$ for $f\in C^{0,\alpha}([a,b])$ and $g\in C^{0,\beta}([a,b])$. So the following lemma holds.
	\begin{lem}\label{lem-young}
		Let $\alpha,\beta>0$, $\alpha+\beta>1$. If $f\in C^{0,\alpha}([a,b])$ and $g\in C^{0,\beta}([a,b])$, then
		\[
		\left|\int\limits_{a}^{b}f(x)dg(x)\right|\leq C(\alpha,\beta,b-a)\|f\|_{C^{0,\alpha}}[g]_{\beta}. 
		\]
		
		If in addition $f(p)=0$ for some $p\in[a,b]$, then
		\[
		\left|\int\limits_{a}^{b}f(x)dg(x)\right|\leq C(\alpha,\beta)|b-a|^{\alpha+\beta}[f]_{\alpha}[g]_{\beta}. 
		\]
		Let $0 < \alpha' \leq \alpha$, $0 < \beta' \leq \beta$ be such that $\alpha' + \beta' > 1$ and
		$\mathrm{Lip}([a,b]) \ni f^{\varepsilon} \to f \text{ in } C^{0,\alpha'}_{\mathrm{b}}([a,b]) \text{ and } \mathrm{Lip}([a,b]) \ni g^{\varepsilon} \to g \text{ in } C^{0,\beta'}_{\mathrm{b}}([a,b]),
		$
		we define 
		\begin{equation*}
			\int\limits_{a}^{b} f  dg := \lim\limits_{\varepsilon \to 0} \int\limits_{a}^{b} f^{\varepsilon}  dg^{\varepsilon} = \lim\limits_{\varepsilon \to 0} \int\limits_{a}^{b} f^{\varepsilon}(x)(g^{\varepsilon})'(x)  dx
		\end{equation*}
		It exists and it does not depend on the choice of the approximation.
	\end{lem}
	Now, we can draw a similar conclusion for the existence of H\"older lift. The proof follows the same two-step pattern as in the Lipschitz and Sobolev cases: necessity is immediate from the commutativity of the diagram together with the fact that the lift closes up on closed curves, while sufficiency requires one to check that the closed-curve condition makes the path-lifting construction well defined, and then to verify that the resulting lift is again $\beta$-H\"older. It is the second step that is genuinely new: the fibre contribution $\int(f\circ\pi_{1})^{*}\alpha_{2}$ must be shown to scale like $\varepsilon^{\beta}$, and this uses the anisotropic estimate of Theorem~\ref{thm-vanish} together with the Young bound of Lemma~\ref{lem-young}. This is the criterion announced as Theorem~\ref{thm-holder-lift-intro}.
	
	\begin{thm}\label{thm-integ}
		Let $U_1 \subset \mathbb{H}$ and $\tilde{U}_1 \subset \pi_1^{-1}(U_1) \subset \mathbb{G}_1$ be domains and $rank(\mathbb{G}_2)=rank(\mathbb{H}_2)$. For $\beta>\frac{1}{2}$, A H\"older map $f\in C^{0,\beta}$, $f: U_1 \to \mathbb{H}_2$ admits a H\"older lift $F\in C^{0,\beta}$, $F: \tilde{U}_1 \to \mathbb{G}_2$ if and only if 
		\begin{equation}\label{integ}
			\int_{f \circ \pi_{1}\circ\widetilde{\gamma}} \alpha = 0
		\end{equation}
		for all Lipschitz closed curves $\widetilde{\gamma}$ on $\mathbb{G}_1$. We define $\int_{\widetilde{\gamma}} (f \circ \pi_{1})^*\alpha:=\lim_{\varepsilon\rightarrow 0}\int_{\widetilde{\gamma}} (f_\varepsilon \circ \pi_{1})^*\alpha$, where  $f_{\varepsilon}=f * \eta_{\varepsilon}=\int_{\mathbb{G}_1 } f(y) \eta_{\varepsilon}\left(y^{-1} x\right) d y$  and $\eta_{\varepsilon}$ is a smooth function on Carnot group $G_1$ with a compact supported set $B(e,\varepsilon)$ and
		\begin{equation*}
			\int_{G_1}\eta_{\varepsilon}=1.
		\end{equation*}
	\end{thm}
	\begin{proof}
		Taking any Lipschitz curve $\widetilde{\gamma}$, the first step is to prove the limit of the integration exists.
		Because \eqref{metric-compar}, all coordinate components of $f$ are $\beta$-H\"older continuous functions. Since $\alpha$ does not extent the horizontal layer and $\beta>\frac{1}{2}$ and theorem \ref{thm-vanish}, we know that the $\beta$-H\"older map $F$ can still preserve the contact structure of Lie group. From \cite{HMS2025} proposition 3.13, the limit of $\int_{\widetilde{\gamma}} (f_\varepsilon \circ \pi_{1})^*\alpha$ exists when $\varepsilon\rightarrow 0$. Thus, we define $\int_{\widetilde{\gamma}} (f \circ \pi_{1})^*\alpha:=\lim_{\varepsilon\rightarrow 0}\int_{\widetilde{\gamma}} (f_\varepsilon \circ \pi_{1})^*\alpha$.
		
		A case in point is Heisenberg group $\mathbb{H}^n$, where the $\alpha=dt-\sum_{i=1}^{n}x_idy_i-y_idx_i$ and $(f \circ \pi_{1})^*\alpha=\sum_{i=1}^{n}f^x_idf^y_i-f^y_idf^x_i$.
		
		For step-3 Carnot group, we know that the space of $\beta$-H\"older continuous function is a Banach algebra and every coordinate component is $\beta$-H\"older continuous function, since 
		\begin{equation*}
			\begin{aligned}
				|f^{B_i}(x)-f^{B_i}(y)|&\lesssim |f(x)-f(y)|\\
				&\leq C d_{cc}(f(x),f(y))\\
				&\lesssim d_{cc}(x,y)^\beta.
			\end{aligned}
		\end{equation*}
		where the constant $C=C(K)$ is from \eqref{metric-compar}, for any $f(x),f(y)\ne \infty$ there is a compact set $K$ such that $f(x),f(y)\in K$.
		Because $rank(\mathbb{G}_2)=rank(H)_2$, $XF_\varepsilon^*(\omega_{2}^{m})\rightarrow 0$ for
		\begin{equation*}
			\omega_{2}^{m}=d C_{m}+\frac{1}{2}\sum_{i=1}^{r} \sum_{j=1}^{s} B_{j} \beta_{m}^{i j}d A_{i}-\frac{1}{12} \sum_{l=1}^{r} \sum_{k=1}^{s}\sum_{i<j} A_{l} \alpha_{k}^{i j}\left( A_{i}d A_{j}- A_{j}d A_{i}\right) \beta_{m}^{l k}. 
		\end{equation*}
		and horizontal vector field $X$.
		Therefore, 
		\begin{equation*}
			\begin{aligned}
				&\frac{1}{2}\sum_{i=1}^{r} \sum_{j=1}^{s}  \beta_{m}^{i j}\int_{a}^{b}\widetilde{\gamma}^B_{j}d \widetilde{\gamma}^A_{i}-\frac{1}{12} \sum_{l=1}^{r} \sum_{k=1}^{s}\sum_{i<j}\alpha_{k}^{i j}\beta_{m}^{l k}\int_{a}^{b} \widetilde{\gamma}^A_{l} \left( \widetilde{\gamma}^A_{i}d \widetilde{\gamma}^A_{j}- \widetilde{\gamma}^A_{j}d \widetilde{\gamma}^A_{i}\right)\\
				&=\widetilde{\gamma}^{C_m}(a)-\widetilde{\gamma}^{C_m}(b)+\int_{\gamma}F^*(\omega_{2}^{m})\\
				&=\widetilde{\gamma}^{C_m}(a)-\widetilde{\gamma}^{C_m}(b)
			\end{aligned}
			,
		\end{equation*}
		where the $\widetilde{\gamma}=F\circ\gamma\in C^{0,\gamma}([0,1];\mathbb{G}_2)$ for Lipschitz curve $\gamma$. 
		
		For higher step situation, the way is the same. 
		
		Due to $F$ is contact lift $\pi_2\circ F \circ \widetilde{\gamma}=f \circ \pi_1\circ \widetilde{\gamma}$, we have $\int_{\widetilde{\gamma}}\left(f \circ \pi_{1}\right)^{*} \alpha_{2}=0$ when $\widetilde{\gamma}$ is a closed Lipschitz curve. 
		
		For the converse direction, because $\int_{f \circ \pi_{1}\circ\widetilde{\gamma}} \alpha = 0$, we can define the lift $F$ in the following sense.
		
		If $\gamma:[0,1]\to H_1$ is a horizontal curve starting from $\gamma(0)=e_{H_1}$, then 
		\begin{equation*}
			\widetilde{\gamma}(1)=\gamma(1) \cdot exp(\int_{\gamma}\alpha)
		\end{equation*}
		For the contact lift $F$ of $f$, then
		\begin{equation*}
			F(\widetilde{\gamma}(1))=f(\gamma(1)) \cdot exp(\int_{\gamma}\left(f \circ \pi_{1}\right)^{*} \alpha_{2})
		\end{equation*}
		Obviously, this $F$ is well defined, since the value of $F$ does not depend on the choice of curve $\gamma$. 
		
		The next step is to prove $F$ is a $\gamma$-H\"older map, that is, to show that the quantity defined by \eqref{integ} obeys the right scaling. Let's take a Lipschitz curve $\widetilde{\gamma}$ connecting $e$ and $\widetilde{\gamma}(1)$ and a Lipschitz curve $\widehat{\gamma}$ connecting $\widetilde{\gamma}(1)=\widehat{\gamma}(0)$ and $\widehat{\gamma}(\varepsilon)$, where $\widehat{\gamma}$ is under arc length parameter.
		\begin{equation*}
			\begin{aligned}
				&d_{CC}(F(\widehat{\gamma}(0)),F(\widehat{\gamma}(\varepsilon)))\\ 
				&\leq d_{CC}(f(\pi_1(\widehat{\gamma}(0))),f(\pi_1(\widehat{\gamma}(\varepsilon))))\\
				&+d_{CC}(e,exp(\int_{\widehat{\gamma}}\left(f \circ \pi_{1}\right)^{*} \alpha_{2}))\\
				&\lesssim \varepsilon^\beta+d_{CC}(e,exp(\int_{\widehat{\gamma}}\left(f \circ \pi_{1}\right)^{*} \alpha_{2}))
			\end{aligned}
		\end{equation*}
		Because 
		\begin{equation*}
			\int_{\widehat{\gamma}}\left(f \circ \pi_{1}\right)^{*} \alpha_{2}=\int_{(f \circ \pi_{1}\circ\widehat{\gamma}(0))^{-1}f \circ \pi_{1}\circ\widehat{\gamma}}\left( r_{f \circ \pi_{1}\circ\widehat{\gamma}(0)}\right)^{*} \alpha_{2}
		\end{equation*}
		where the $r$ is the right translation and $(f \circ \pi_{1}\circ\widehat{\gamma}(0))^{-1}f \circ \pi_{1}\circ\widehat{\gamma}$ is also a $\gamma$-H\"older curve, for convenience, we can assume $f \circ \pi_{1}\circ\widehat{\gamma}(0)=e$.
		
		We take step 3 Carnot group as an example.
		
		From \eqref{metric-K},
		\begin{equation*}
			\begin{aligned}
				|(f \circ \pi_{1}\circ\widehat{\gamma})^{B_i}(\varepsilon)-(f \circ \pi_{1}\circ\widehat{\gamma})^{B_i}(0)|^{\frac{1}{2}}&\leq d_K((f \circ \pi_{1}\circ\widehat{\gamma})(\varepsilon),(f \circ \pi_{1}\circ\widehat{\gamma})(0))\\
				&\approx d_{CC}(e,f \circ \pi_{1}\circ\widehat{\gamma}(\varepsilon))\\
				&\lesssim \varepsilon^\beta 
			\end{aligned}
		\end{equation*}
		thus the coordinate component $(f \circ \pi_{1}\circ\widehat{\gamma})^{B_i}$ is a $\varepsilon^{2\beta}$-H\"older function.
		
		By the result form lemma 5.8 in  \cite{HMS2025} yields 
		\begin{equation*}
			\int_{\widetilde{\gamma}} (f \circ \pi_{1})^*\alpha\lesssim \varepsilon^{3\beta}[f\circ \pi_{1}\circ\widetilde{\gamma}]_\gamma^2
		\end{equation*}
		Thus, 
		\begin{equation*}
			\begin{aligned}
				&d_{CC}(F(\widehat{\gamma}(0)),F(\widehat{\gamma}(\varepsilon)))\\ 
				&\lesssim \varepsilon^\beta+d_{CC}(e,exp(\int_{\widehat{\gamma}}\left(f \circ \pi_{1}\right)^{*} \alpha_{2}))\\
				&\approx \varepsilon^\beta+\left|\int_{\widetilde{\gamma}} (f \circ \pi_{1})^*\alpha\right|^{\frac{1}{3}}\\
				&\lesssim \varepsilon^\beta+\varepsilon^\beta\\
				&=2\varepsilon^\beta
			\end{aligned}
		\end{equation*}
		Therefore, the lift $F$ is a $\beta$-H\"older map. For the higher step Carnot group, the way is the same.
	\end{proof}
	
	In \cite{HHI2025}, theorem 1.1 and proposition 1.4 give decent results about the existence of contact lifts for differentiable maps. However, for H\"older map, the result about the existence of lift can just rely on the metric structure. Compared with \cite{HHI2025}, in this section we just consider a special case of central extension.
	
	Theorem~\ref{thm-integ} is conceptually clean, but it asks one to verify a condition on every Lipschitz closed curve, which is hardly checkable in practice. The following criterion replaces it by a condition involving only horizontal curves, a single linear map between the fibres, and a correction term of higher weight; it is the version used in Example~\ref{ex-weierstrass} below, and it is Theorem~\ref{thm-criterion-intro} of the introduction.
	
	\begin{thm}\label{thm-criti}
		Let $V_1 \to G_1 \to H_1$ and $V_2 \to G_2 \to H_2$ be central extensions, where $rank(H_2)=rank(G_2)$. Let $U$ is a simply connected domain in $H_1$ and $f: U \to H_2$ be a $\gamma$-H\"older map for $\beta>\frac{1}{2}$. Then $f$ admits a $\beta$-H\"older lift $F$ if and only if there is a linear map $\varphi: V_1 \to V_2$ and a $V_2$-value 1-form $\lambda$ with weight $\geq 2 $ such that the integral of $f^*\alpha_{2}$ on any horizontal curve equals to the integral of $\varphi\circ \alpha_1$ on the same curve, i.e.   
		\begin{equation*}
			\varphi\circ \int_{\gamma} \alpha_1=  \int_{\gamma} f^*\alpha_{2}.
		\end{equation*}
		such that for any horizontal curve. 
		
		\begin{displaymath}
			\xymatrix{
				V_1 \ar[r] \ar[d]_\varphi & G_1 \ar[d]^{F} \ar[r]^{\pi_1} & H_1 \ar[d]^{f}
				\\
				V_2 \ar[r] & G_2  \ar[r]^{\pi_2} & H_2
			}
		\end{displaymath} 
	\end{thm}
	\begin{proof}
		If $f$ admits a $\beta$-H\"older lift $F$, thanks to theorem \ref{thm-decomp}, there is a linear map $\varphi: V_1 \to V_2$. Because of $\beta>\frac{1}{2}$ and theorem \ref{thm-vanish}, we know that the local $\beta$-H\"older map $F$ can still preserve the contact structure of Lie group.
		\begin{equation*}
			\begin{aligned}
				F(hk)&=F(h)\Phi(k)\\
				&=\pi_{V_2}\circ F(h)\cdot \pi_2\circ F(h) \cdot \Phi(k)\\
				&=C_{V_2}\cdot \pi_2\circ F(h) \cdot \Phi(k)\\
				&=C_{V_2}\cdot f\circ \pi_1(h) \cdot \Phi(k)\\
				&=C_{V_2}\cdot f(h) \Phi(k)\\
				&=f(h) \cdot C_{V_2}\cdot \Phi(k)
			\end{aligned}
		\end{equation*}
		for $h\in H$ and $k\in V_1$. The third equation is because $\pi_{V_2}\circ F|_H:H\to V_2$ is a local $\beta$-H\"older map and condition $rank(H_2)=rank(G_2)$, the only $\beta$-H\"older is constant mapping with the same proof as in theorem \ref{thm-decomp}. Thus, 
		\begin{equation*}
			\varphi\circ \int_{\gamma} \alpha_1= C_{V_2}\cdot \Phi(k)=  \int_{\gamma} (f\circ\pi_{1})^*\alpha_{2},
		\end{equation*}
		where we take $\gamma$ as the geodesic from $e$ to $k$. 
		Obviously, here the $\gamma$ can be replaced by any Lipschitz curve.
		
		For convenience, we can denote as the following form. There is a $V_2$-value 1-form $\lambda$ with weight $\geq2 $ such that 
		\begin{equation*}
			\varphi\circ \alpha_1=  (f\circ\pi_{1})^*\alpha_{2}+\lambda.
		\end{equation*}
		The $(f\circ\pi_{1})^*\alpha_{2}$ can be decomposed into two part: $f^*\alpha_{2} :V_1\to V_2$ and $\mu_{h}:H\to V_2$.
		For the $\varphi\circ \alpha_1$, it turns  $\alpha_1$ with weight $=n+1 $ to a form with weight $=n+1$. For $f^*\alpha_{2}$, it turns $\alpha_2$ with weight $=n+1 $ to a form with weight $=n+1$. However, the $\mu_{h}=0$, since it turns a 1-form with weight $<n+1 $ to a form with weight $=n+1$. 
		
		Therefore, 
		\begin{equation*}
			\begin{aligned}
				\varphi\circ \alpha_1&=  f^*\alpha_{2}+\mu_{h}+\lambda\\
				&=f^*\alpha_{2}+\lambda.
			\end{aligned}
		\end{equation*}
		and
		\begin{equation*}
			\varphi\circ \int_{\gamma} \alpha_1=  \int_{\gamma} f^*\alpha_{2}.
		\end{equation*}
		
		For the converse direction, because distributionally $\varphi\circ\alpha_{1}=f^*\alpha_{2}+\lambda$  for any horizontal curve, locally for any fixed point $p$ and its closed cubic neighbor $U(p)$
		\begin{equation*}
			\int_{\widetilde{\gamma}} f^*\alpha_2 = \varphi\circ \int_{\widetilde{\gamma}} \alpha_1
		\end{equation*}
		where $\widetilde{\gamma}\subset U(p) $. That means the value of $F$ dose not depend on the choice of path in $U(p)$.  We can divide the area into cubes $\left\{ U(p)\right\}$. Since the domain $U$ is simply connected, when considering the value of $\int_{\widetilde{\gamma}} f^*\alpha_2$, any given horizontal curve $\widetilde{\gamma}$ is equivalent to a curve near itself. Eventually, this curve can be moved to arbitrary horizontal curve. Therefore, for any closed Lipschitz curve $\widehat{\gamma}$
		\begin{equation*}
			\int_{\widehat{\gamma}} f^*\alpha_2 = \varphi\circ \int_{\widehat{\gamma}} \alpha_1=0.
		\end{equation*}
		Then $f$ admits a $\beta$-H\"older lift $F$, since theorem \ref{thm-integ}.
	\end{proof}
	
	\noindent\textbf{Example~\ref{ex-weierstrass}.}\label{ex-weierstrass} For Heisenberg group $H_1$,
	We can take $f(x,y)=(ax+g(y),\frac{1}{a^2}y):\mathbb{R}^2\to\mathbb{R}^2$, where $g$ is a Weierstrass $\frac{2}{3}$-H\"older function
	\begin{equation*}
		g(y)=\sum_{n=0}^{\infty}{\frac{1}{9}}^n cos(27^n\pi y)
	\end{equation*}
	It is obvious that $f$ admits a contact lift, since
	\begin{equation*}
		\begin{aligned}
			f^*(dx\wedge dy)&=\lim_{\varepsilon} f_\varepsilon^*(dx\wedge dy)\\
			&=\lim_{\varepsilon} f_\varepsilon^*(dx)\wedge f^*(dy)\\
			&=\lim_{\varepsilon} \left(\frac{\partial}{\partial x}f_\varepsilon^xdx+\frac{\partial}{\partial y}f_\varepsilon^xdy\right)\wedge \frac{1}{a^2}dy\\
			&=\lim_{\varepsilon} \left(\frac{\partial}{\partial x}f^x*\eta_\varepsilon dx+\frac{\partial}{\partial y}f^x*\eta_\varepsilon dy\right)\wedge \frac{1}{a^2}dy\\
			&=\lim_{\varepsilon} \frac{1}{a^2} \frac{\partial}{\partial x}f^x*\eta_\varepsilon dx \wedge dy\\
			&=\lim_{\varepsilon} \frac{1}{a} dx \wedge dy\\
			&=\frac{1}{a} dx \wedge dy.
		\end{aligned}
	\end{equation*}
	Therefore, with the help of theorem \ref{thm-criti}, we can find a $\frac{2}{3}$-H\"older lift $F:H_1\to H_1$.
	
	The central extension is defined by $d_0\alpha=\rho=dx\wedge dy$ as
	\begin{displaymath}
		\xymatrix{
			R \ar[r] & H_1 \ar[d]^{F} \ar[r]^{\pi} & \mathbb{R}^2 \ar[d]^{f}
			\\
			R \ar[r] & H_1  \ar[r]^{\pi} & \mathbb{R}^2
		}
	\end{displaymath}
	
	If $\gamma:[0,1]\to\mathbb{R}^2$ is a Lipschitz curve starting from $\gamma(0)=e$, then 
	\begin{equation*}
		\widetilde{\gamma}(1)=\gamma(1) \cdot exp(\int_{\gamma}\alpha)
	\end{equation*}
	and
	\begin{equation*}
		F(\widetilde{\gamma}(1))=f(\gamma(1)) \cdot exp(\int_{\gamma}\left(f \circ \pi\right)^{*} \alpha)
	\end{equation*}

	\nocite{BellaicheRisler1996}
	\bibliographystyle{plain}
	\bibliography{ref}
\end{document}